\documentclass[10pt]{amsart}
\usepackage{amsthm}
\usepackage{amstext}
\usepackage{amssymb}

\numberwithin{equation}{section}
\numberwithin{figure}{section}

\theoremstyle{definition}
\newtheorem*{defn*}{\protect\definitionname}
\theoremstyle{remark}
\newtheorem*{rem*}{\protect\remarkname}
\theoremstyle{plain}
\newtheorem*{thma*}{Theorem A}
\newtheorem*{thmb*}{Theorem B}
\newtheorem*{thmc*}{Theorem C}
\theoremstyle{plain}
\newtheorem*{claim}{Claim}
\newtheorem{thm}{\protect\theoremname}[section]
\theoremstyle{plain}
\newtheorem{prop}[thm]{\protect\propositionname}
\theoremstyle{plain}
\newtheorem{lem}[thm]{\protect\lemmaname}
\theoremstyle{remark}
\newtheorem{rem}[thm]{\protect\remarkname}
\theoremstyle{plain}
\newtheorem{cor}[thm]{\protect\corollaryname}
\theoremstyle{definition}
\newtheorem{example}[thm]{\protect\examplename}
\newtheorem{defn}[thm]{Definition}

\providecommand{\corollaryname}{Corollary}
\providecommand{\definitionname}{Definition}
\providecommand{\examplename}{Example}
\providecommand{\lemmaname}{Lemma}
\providecommand{\propositionname}{Proposition}
\providecommand{\remarkname}{Remark}
\providecommand{\theoremname}{Theorem}
\providecommand{\theoremname}{Theorem}

\newcommand{\abs}[1]{\left\vert#1\right\vert}
\newcommand{\set}[1]{\left\{#1\right\}}
\newcommand{\Real}{\mathbb{R}}
\newcommand{\sph}{\mathbb{S}}

\newcommand{\bmu}{\bar{\mu}}
\newcommand{\Hess}{\mathrm{Hess}}
\newcommand{\scal}{\mathrm{scal}}
\newcommand{\Ric}{\mathrm{Ric}}

\begin{document}

\title{Warped Product Rigidity}

\author{Chenxu He}
\address{14 E. Packer Ave\\
Dept. of Math, Lehigh University \\
Christmas-Saucon Hall\\
Bethlehem, PA, 18015.}
\curraddr{Department of Math, University of Oklahoma \\
Norman, OK 73019.}
\email{he.chenxu@lehigh.edu}
\urladdr{http://sites.google.com/site/hechenxu/}

\author{Peter Petersen}
\address{520 Portola Plaza\\
Dept of Math UCLA\\
Los Angeles, CA 90095}
\email{petersen@math.ucla.edu}
\urladdr{http://www.math.ucla.edu/\textasciitilde{}petersen}
\thanks{The second author was supported in part by NSF-DMS grant 1006677}

\author{William Wylie}
\address{215 Carnegie Building\\
Dept. of Math, Syracuse University\\
Syracuse, NY, 13244.}
\email{wwylie@syr.edu}
\urladdr{http://wwylie.mysite.syr.edu}
\thanks{The third author was supported in part by NSF-DMS grant 0905527 }

\dedicatory{In memory of Barrett O'Neill}


\begin{abstract}
In this paper we study the space of solutions to an overdetermined linear system involving the Hessian of functions. We show that if the solution space has dimension greater than one, then the underlying manifold has a very rigid warped product structure. We obtain a uniqueness result for prescribing the Ricci curvature of a warped product manifold over a fixed base. As an application, this warped product structure will be used to study warped product Einstein structures in \cite{HPWVirt-Soln}.
\end{abstract}
\maketitle

\section*{Introduction}

Let $q$ be a quadratic form on a Riemannian manifold $\left(M,g\right)$ and $Q$ be the corresponding symmetric linear operator on $M$. We shall study the space of solutions
\[
W\left(M;q\right)=\left\{ w\in C^{\infty}\left(M,\mathbb{R}\right):\mathrm{Hess}w=wq\right\} .
\]
Solving $\mathrm{Hess}w=wq$ for a fixed $q$ is generally impossible, but it is a problem that appears in many places.  Perhaps the best known example is due to Obata \cite{Obata} which we will discuss in section 1.  More complicated examples are warped product Einstein structures which  are of this type with
\[
q=\frac{1}{m}\left(\mathrm{Ric}-\lambda g\right)
\]
see \cite{HPWconstantscal} and \cite{HPWVirt-Soln}.

For any positive function $w$,
\[
\mathrm{Hess}w= w q
\]
defines a quadratic form $q$ such that $w\in W(M; q)$. However, if a real valued function $w$ satisfies such an equation, then its zero set is a totally geodesic codimension one submanifold, which is a rather special condition. We shall enhance this by showing that, if such an equation has linearly independent solutions, then the underlying space is a warped product.

Note that when $\dim M=1$ the equation
\[
\mathrm{Hess}w=wq
\]
is a scalar equation
\[
w^{\prime\prime}=Qw
\]
with $q=Qdx^{2}$. Clearly there is a two-dimensional space of solutions unless $M=\sph^{1}$. So this is not a case where we can say much about $\left(M,g,q\right)$. When $M=\sph^{1}$ this equation is also known as Hill's equation. The issue of finding one or two solutions to that equation has a long history (see \cite{Magnus-Winkler}.) In either case the underlying space does have the desired underlying structure of a warped product, albeit in a very trivial fashion with the base being a point and the fiber the space itself. This example shows that one cannot expect $q$ to be determined by the geometry unless there are three or more linearly independent solutions.

The building blocks for all examples consist of base spaces and fiber spaces:
\begin{defn*}
A \emph{base space} $\left(B,g_{B},u\right)$ consists of a Riemannian manifold and a smooth function $u:B\rightarrow[0,\infty)$ such that $u^{-1}\left(0\right)=\partial B$. We define
\[
q_{B}=\frac{1}{u}\mathrm{Hess}u
\]
and when $\partial B\neq\emptyset$ assume that this defines a smooth tensor on $B$, and that $\left|\nabla u\right|=1$ on $\partial B$. Moreover, if the functions in the solution space $W\left(B;q_{B}\right)$, that vanish on $\partial B$ when it is not empty, are constant multiples of $u$, then we call $(B, g_B, u)$ a \emph{base manifold} (see \cite{HPWVirt-Soln}).
\end{defn*}
\smallskip{}

\begin{defn*}
A \emph{fiber space} $\left(F,g_{F},\tau\right)$ consists of a space form $\left(F,g_{F}\right)$ and a \emph{characteristic function} $\tau:F\rightarrow\mathbb{R}$ such that $\dim W\left(F;-\tau g_{F}\right)=\dim F+1$. In case $F$ is a sphere we shall further assume that $\left(F,g_{F}\right)$ is the unit sphere.
\end{defn*}

\begin{rem*}
As we shall see $\tau$ will almost always be a constant. Only when $\dim F=1$ is it possible for $\tau$ to be a function. We shall also see that $F$ must be simply connected unless it is a circle.
\end{rem*}

Our first result is that if $W(M;q)$ has dimension larger than one then $(M,g)$ must  be isometric to a warped product of a particular sort.

\begin{thma*}
Let $\left(M,g\right)$ be complete and simply-connected. If $q$ is a quadratic form such that $\dim W\left(M;q\right)=k+1$ where $k\geq1$, then there is a simply connected base space $\left(B,g_{B},u\right)$ and a fiber space $\left(F,g_{F},\tau\right)$ such that
\[
\left(M,g\right)=\left(B\times F,g_{B}+u^{2}g_{F}\right).
\]
Moreover when $\partial B\neq\emptyset$ or $k>1$, then the characteristic function is constant.
\end{thma*}

\begin{rem*} In general the base space doesn't have to be a base manifold, see Example \ref{Exa:B=R}. This is in sharp contrast to what happens when $q$ is more directly related to the geometry (see \cite{HPWVirt-Soln}).\end{rem*}

From the warped product structure $M = B\times_u F$  constructed in Theorem A,  we obtain  two natural projections $\pi_1$ and $\pi_2$ from $M$ to $B$ and $F$ respectively. The special structure of the manifold $M$ yields the following decomposition of the vector space $W(M; q)$.
\begin{thmb*}
Let $(M, g)$ be complete and simply-connected with $\dim W(M; q) \geq 2$. Suppose $M = B\times_u F$ as in Theorem A. Then we have
\[
W(M; q) = \set{\pi_1^*(u)\cdot \pi_2^*(v) : v \in W(F; -\tau g_F)}.
\]
\end{thmb*}

\smallskip

\begin{rem*}
We actually show more general results than Theorems A and B. Namely any subspace $W \subset W(M; q)$ with $\dim W > 1$, not necessarily the whole solution space, induces a warped product structure on $M$ as in Theorem A and such $W$ has the decomposition as in Theorem B, see Theorems \ref{thm:WPHWP}, \ref{thm:W-S-notempty} and \ref{thm:W}. This generalization will be useful when we study manifolds that are not simply-connected, see Proposition \ref{prop:nonsimplyconn}. By lifting the quadratic form $q$ to the universal cover we consider the subspace of solutions that are also invariant under deck transformations.
\end{rem*}

\begin{rem*}
The special type of warped product obtained in Theorem A with constant curvature fiber is also referred as a \emph{generalized Robertson-Walker} space in general relativity, see the recent survey \cite{Zeghib} and references therein.
\end{rem*}

The other main result, which is an application of Theorems A and B, is a uniqueness result on warped product metrics.

\begin{thmc*}
Let $(M^{n},g)$ be a complete Riemannian manifold and $w_{1}$, $w_{2}$
 two positive functions on $M$. Let $(N_{1}^{d},h_{1})$ and $(N_{2}^{d},h_{2})$
be two simply-connected space forms. Suppose $(E_{i}, g_{E_{i}}) = \left(M\times N_{i}, g_M + w_i^2 h_i\right)(i=1,2)$ are two warped product metrics having the same scalar curvature as functions on $M$. Furthermore, assume that the Ricci tensors of $E_{1}$ and $E_{2}$ when restricted to the subbundle $TM\subset TE_{i}$ are the same. Then either
\begin{enumerate}
\item $E_{1}$ and $E_{2}$ are isometric to each other, or
\item $d=1$ and $E_1$, $E_2$ are warped product extensions of a base space $B^{n-1}$ with non-isometric isocurved fibers $(\Real\times_{w_i} N_i,dt^2 + w_i^2(t)h_i)(i=1,2)$.
\end{enumerate}
In particular, if $M$ is compact then $E_1$ and $E_2$ are isometric.
\end{thmc*}

\begin{rem*}
On a warped product manifold $E = M^n \times_{w}N^d$ with metric $g_E = g_M + w^2 g_N$, if the scalar curvature of the fiber $(N, g_N)$ is constant, then the scalar curvature of $E$ is constant along each fiber, i.e., defines a function on the base $M$. See \cite[p. 214]{O'Neill}.
\end{rem*}

\begin{rem*}
The assumption that $(N_i, h_i)$ is a simply-connected space form is necessary. Otherwise one can replace one of $N_i$'s by another manifold with the same scalar curvature and then the isometry between $E_1$ and $E_2$ fails completely.
\end{rem*}

\begin{rem*}
Two Riemannian manifolds $(M, g)$ and $(\bar{M}, \bar{g})$ are called \emph{isocurved} if there is a diffeomorphism $f: M \rightarrow \bar{M}$ that preserves the sectional curvatures, i.e., for any $x\in M$ and any two-plane $P \subset T_x M$ we have
\[
K(P) = \bar{K} (f_* P).
\]
In \cite{Kulkarni} R. S. Kulkarni showed that, if the dimension is larger than three, then isocurved metrics are isometric except for the coverings between space forms, also see \cite{Yau}. It is well known that this is not true in dimension two.  It is also important to keep in mind that the metrics on $E_1$ and $E_2$ in case (2)  will generally not be isocurved,  as the property of being isocurved is not preserved by taking products or warped products.
\end{rem*}

\begin{rem*}
In Example \ref{ex:surfacenonisometric} we construct two complete isocurved metrics of the form
\[
(\Real \times_{w_i} \Real , dt^2 + w_{i}^2(t) dx^2)
\]
which are not isometric.  This shows that there are metrics (with $B$ a point)  in case (2) which are not isometric. From Remark \ref{rem:surfacenonisometric}, these examples can also be used to construct metrics of any dimension which also fall into case (2) and are not isometric. On the other hand, if $E_1$ and $E_2$ have some additional geometric structure, for example they are Einstein manifolds or gradient Ricci solitons, we are able to show that this case does not exist, see \cite{HPWVirt-Soln,HPWuniqsoliton}.
\end{rem*}

\begin{rem*}
It remains an open question to characterize when there is an isometry between the metrics falling into case (2).  On the other hand, the isometry in case (1) between $E_1$ and $E_2$ preserves the corresponding warped product splitting, and there can be no such isometry in case (2).
\end{rem*}

In Section 1 we show some basic properties about the solution space $W\left(M;q\right)$ and what it looks like in some simple cases. In Section 2 we establish some properties for $W\left(M;q\right)$ when we know that $M$ is a warped product. Section 3 is devoted to the proof of Theorem A. In Section 4 we use Theorem A to place restrictions on what the quadratic form can look like. This in turn is used in Section 5 to prove Theorem B. Knowing that $M$ is a warped product then allows us to determine what the quadratic form $q$ looks like in terms of the geometry of $M$. In section 6 we consider the manifold $M$ which may not be simply-connected. Theorems A and B are also valid unless the metric has a very special form, see Proposition \ref{prop:nonsimplyconn}. In Section 7, we prove Theorem C. In Section 8 we collect some miscellaneous results: more detailed description of the base space that appears in the warped product structure and, a natural Lie algebra structure on the exterior square $\wedge^2 W$.

\medskip

\textbf{Acknowledgment:} The authors would like to thank David Johnson for helpful discussions.

\medskip
\section{Basic Properties and Examples}

We start by establishing two elementary but fundamental properties.

\begin{prop}\label{prop:injectionHWP}
The evaluation map
\begin{eqnarray*}
W(M;q) & \rightarrow & \mathbb{R}\times T_{p}M,\\
w & \mapsto & \left(w\left(p\right),\nabla w|_{p}\right)
\end{eqnarray*}
is injective.
\end{prop}

\begin{proof}
We use a proof adapted from \cite{Corvino} and which was also used in \cite{HPWLCF}. Let $w\in W=W(M;q)$ and let $\gamma$ be a unit speed geodesic emanating from $p\in M$. Define $h(t)=w(\gamma(t))$ and $\Theta(t)=q\left(\gamma^{\prime}(t),\gamma^{\prime}(t)\right)$. Then we have a linear second order o.d.e. along $\gamma$ for $h$ given by
\begin{eqnarray*}
h^{\prime\prime}(t) & = & \mathrm{Hess}w(\gamma^{\prime}(t),\gamma^{\prime}(t))\\
& = & \frac{\Theta(t)}{m}\cdot h(t).
\end{eqnarray*}
Thus $h$ is uniquely determined by its initial values
\begin{eqnarray*}
h\left(0\right) & = & w\left(\gamma\left(0\right)\right),\\
h^{\prime}\left(0\right) & = & g\left(\nabla w,\gamma^{\prime}\left(0\right)\right).
\end{eqnarray*}
In particular $h$ must vanish if $w$ and its gradient vanish at $p.$

This shows that the set $A=\left\{ p\in M:w\left(p\right)=0,\nabla w|_{p}=0\right\} $ is open. As it is clearly also closed it follows that $w$ must vanish everywhere if $M$ is connected and $A$ is nonempty.
\end{proof}

Next we prove a basic fact about the zero set of a $w\in W(M;q)$.
\begin{prop}\label{prop:hypersurfaceHWP}
Let $L=\left\{ p\in M:w(p)=0\right\} \neq\emptyset$ for some $w\in W(M;q).$ Then $L$ is a totally geodesic hypersurface.
\end{prop}

\begin{proof}
We already know that $\nabla w$ can't vanish on $L.$ This shows that 0 is a regular value for $w$ and hence that $L$ is hypersurface. We know in addition from $\nabla_{X}\nabla w=wQ\left(X\right)$ that $\mathrm{Hess}w$ vanishes on $L.$ This shows that $L$ is totally geodesic.
\end{proof}

 Recall that Obata  \cite{Obata}  proved  (in our notation) that a complete metric $(M,g)$  supports a non-constant function  in $W\left(M;-g\right)$
if and only if   $(M,g)$ is isometric to the sphere.  Many other authors have proven different generalizations of this result.  In particular, Tashiro \cite{Tashiro} showed that a complete metric $(M,g)$ has  a non-constant function in $W\left(M;-\tau g\right)$ for some function $\tau$ if and only if $(M,g)$ is (globally) of the form,
\[ dt^2 + (v^2(t)) g_N.  \]
In fact, a local version of this result was used by Brinkmann in the 1920s \cite{Brinkmann}.  Also see \cite{Cheeger-Colding}, \cite{OS}, and \cite[9.117]{Besse}.  See \cite{KuehnelRademacher} for  the result in the pseudo-Riemannian case.

We extend these results by studying the case where $\dim W\left(M;-\tau g\right)$ is maximal.

\begin{thm}
Let $\left(M,g\right)$ be a complete simply connected Riemannian $n$-manifold with $n>1$. If there exists $\tau\in\mathbb{R}$ such that
\[
\dim W\left(M;-\tau g\right)=n+1,
\]
then $\left(M,g\right)$ is a simply connected space form of constant curvature $\tau$.
\end{thm}

\begin{proof}
Obata considered the case where $\tau>0$ and in that case it suffices to assume that $\dim W\left(M;-\tau g\right)\geq1$. In case $\tau\leq0$ we do however need the stronger assumption.

When $\tau=0$ we see that constant functions are in $W\left(M;0\right)$. But there will also be an $n$-dimensional subspace of non-constant functions whose Hessians vanish. This shows that $\left(M,g\right)=\mathbb{R}^{n}$ with the Euclidean flat metric.

When $\tau<0$ note that any $w\in W\left(M;-\tau g\right)$ has the property that $\bar{\mu}\left(w\right)=\tau w^{2}+\left|\nabla w\right|^{2}$ is constant and thus defines a nondegenerate quadratic form on $W\left(M;-\tau g\right)$. By Proposition \ref{prop:injectionHWP} it follows that some $w$ will have $\bar{\mu}\left(w\right)<0$. By scaling we can then assume that some $w\in W\left(M;-\tau g\right)$ will satisfy $\tau w^{2}+\left|\nabla w\right|^{2}=\tau$ or
\[
-\tau=\frac{\left|\nabla w\right|^{2}}{-1+w^{2}}=\left|\nabla\mathrm{arccosh}\left(w\right)\right|^{2}.
\]
By \cite{Tashiro},   $w=\cosh\left(\sqrt{-\tau}r\right)$, where $r:M\rightarrow\mathbb{R}$ is the distance to a point in $M$, and the metric has constant curvature $\tau$ as it is the warped product metric
\[
dr^{2}+\frac{\sinh\left(\sqrt{-\tau}r\right)}{\sqrt{-\tau}}g_{\mathbb{S}^{n-1}}
\]
where $g_{\mathbb{S}^{n-1}}$ is the standard metric on the unit sphere.
\end{proof}

This result can be further extended as follows:
\begin{lem}\label{lem:GenObata}
Let $\left(M,g\right)$ be a complete simply connected Riemannian $n$-manifold with $n\geq1$. If there exists $\tau:M\rightarrow\mathbb{R}$ such that
\[
\dim W\left(M;-\tau g\right)=n+1,
\]
then either $n=1$ or $\tau$ is constant and $\left(M,g\right)$ is a simply connected space form of constant curvature $\tau$.
\end{lem}

\begin{proof}
When $n=1$ it is clear that $\dim W\left(M;-\tau g\right)=2$ for any function $\tau$. In general having a solution to
\[
\mathrm{Hess}w=-\tau wg
\]
shows that
\[
d\left(\left|\nabla w\right|^{2}\right)=-\tau d\left(w^{2}\right)
\]
In particular, $d\tau\wedge d\left(w^{2}\right)=0$ for all $w\in W\left(M;-\tau g\right)$. Now use $\dim W\left(M;-\tau g\right)=n+1$ together with Proposition \ref{prop:injectionHWP} to find $n$ functions $w_{i}\in W\left(M;-\tau g\right),\, i=1,...,n$ such that $d\left(w^{2}\right)$ form a basis at $x$. Then $d\tau\wedge d\left(w_{i}^{2}\right)=0$ implies that $d\tau$ vanishes at $x$.
\end{proof}

In case $M$ has constant curvature we also have the following converse.
\begin{lem}\label{lem:ReverseObata}
Let $\left(M,g\right)$ be a complete simply connected space form and $\tau\in\mathbb{R}$. Either $\dim W\left(M;-\tau g\right)=\dim M+1$ or all functions in $W\left(M;-\tau g\right)$ are constant.
\end{lem}

\begin{proof}
When $n=1$ this is obvious. Otherwise we have to show that if $w$ is a non-constant solution to the equation $\mathrm{Hess}w=-\tau wg$, then $\tau$ is the curvature of $M$. However, the fact that $\mathrm{Hess}w=-\tau wg$ together with knowing that $\tau w^{2}+\left|\nabla w\right|^{2}$ is constant shows that $\mathrm{sec}\left(\nabla w,X\right)=\tau$.
\end{proof}

\smallskip

In Theorem A, even though the manifolds $M$ do not have boundary, manifolds with boundary do arise as base spaces. The warped product structure of $M$ yields a decomposition of the space $W(M; q)$ involving  functions in the space $W(B; q|_{B})$, see Proposition \ref{prop:WsplittingHWP}.  When $B$ has boundary,   the boundary conditions satisfied by these functions  in $W(B; q|_{B})$ also  yields some further restrictions. We will encounter both Dirichlet and Neumman boundary conditions, for which we will use the following notation.

\begin{defn}
Let $M$ be a Riemannian manifold with boundary $\partial M \ne \emptyset$ and $\nu$ is a normal vector to $\partial M$, then we define the spaces with Dirichlet and Neumann boundary conditions as
\begin{eqnarray*}
W(M; q) & = & \set{w \in C^{\infty}(M) : \mathrm{Hess} w = w q}, \\
W(M; q)_D & = & \set{w \in W(M; q) : w|_{\partial M} = 0}, \\
W(M; q)_N & = & \set{w \in W(M; q) : {\frac{\partial w}{\partial \nu}}|_{\partial M} = 0}.
\end{eqnarray*}
\end{defn}

We have the following general fact when $\partial M \ne \emptyset$.

\begin{prop}
If $\partial M \ne \emptyset$, then $\dim W(M; q)_D \leq 1$. Moreover if $\dim W(M; q)_D = 1$, then
\[
W(M; q) = W(M; q)_D \oplus W(M; q)_N.
\]
\end{prop}

\begin{proof}
Let $x \in \partial M$ and $w\in W(M; q)_D$. If $\nabla w(x) = 0$, then Proposition \ref{prop:injectionHWP} implies that $w$ is the zero function. Thus any non-zero element $w \in W(M; q)_D$ satisfies
\[
w(x) =0 \quad \nabla w(x) \ne 0 \quad \nabla w(x)\perp \partial B
\]
which, applying Proposition \ref{prop:injectionHWP} again shows that $\dim W(M; q)_D \leq 1$.

In the case when $\dim W(M; q)_D = 1$ Proposition \ref{prop:injectionHWP} also shows that the intersection is zero,
\[
W(M;q)_D \cap W(M;q)_N = \set{0},
\]
i.e., the decomposition of $W(M;q)$ is a direct sum.
\end{proof}
\begin{rem}
Note that the condition $\dim W(M; q)_D = 1$ implies that $M$ has totally geodesic boundary.
\end{rem}

In the following we describe some basic examples, where the manifold $M$ is one dimensional. We consider the quadratic form $q$ defined by the warped product Einstein equation, i.e.,
\[
q = -\tau g \quad \text{with}\quad \tau = \frac{\lambda}{m}.
\]
Note that $\mathrm{Ric} = 0$ in this case. In \cite[Example 1]{HPWLCF} we had the classification when the solution $w \in W(M; -\tau g)$ is non-negative and only vanishes on the boundary. Here we extend that classification to allow any sign of $w$.

\begin{example}
Let $(M, g) = (\Real, dt^2)$, then $w \in W(M; -\tau g)$ if and only if
\[
w'' = -\tau w.
\]
So we have three different cases depending on the sign of $\lambda$.
\begin{enumerate}
\item When $\lambda > 0$, we have $w = C_1 \cos(\sqrt{\tau}t) + C_2 \sin(\sqrt{\tau}t)$.
\item When $\lambda = 0$, we have $w = C_1 t + C_2$.
\item When $\lambda < 0$, we have $w = C_1 \exp(\sqrt{-\tau} t) + C_2 \exp(-\sqrt{-\tau} t)$.
\end{enumerate}
In particular we have $\dim W(M; - \tau g) = 2$ for all three cases.
\end{example}

\begin{example}
Let $(M, g) = (\sph^1_a, dt^2)$ be the circle with radius $a$. Then $W(\sph^1_a; -\tau g)$ corresponds to the elements in $W(\Real; -\tau g)$ which have period $2\pi a$. This gives us the following
\[
\dim W(\sph^1_a; -\tau g) = \left\{
\begin{array}{rl}
2 & \text{ if } \lambda > 0 \text{ and } a\sqrt{\tau} \text{ is an integer,} \\
1 & \text{ if } \lambda = 0, \\
0 & \text{ otherwise.}
\end{array}
\right.
\]
\end{example}

We now look at one dimensional examples with boundary.

\begin{example}
Let $(M, g) = ([0,\infty), dt^2)$ and then again we have three different cases.
\begin{enumerate}
\item When $\lambda > 0$, we have
\begin{eqnarray*}
W(M; -\tau g)_D & = & \set{C \sin (\sqrt{\tau}t) : C \in \Real} \\
W(M; -\tau g)_N & = & \set{C \cos(\sqrt{\tau}t) : C \in \Real}.
\end{eqnarray*}
\item When $\lambda = 0$, we have
\begin{eqnarray*}
W(M; -\tau g)_D & = & \set{C t : C \in \Real} \\
W(M; -\tau g)_N & = & \set{C : C \in \Real}.
\end{eqnarray*}
\item When $\lambda < 0$, we have
\begin{eqnarray*}
W(M; -\tau g)_D & = & \set{C \sinh (\sqrt{-\tau}t) : C \in \Real} \\
W(M; -\tau g)_N & = & \set{C \cosh(\sqrt{-\tau}t) : C \in \Real}.
\end{eqnarray*}
\end{enumerate}
\end{example}

Finally we consider the closed interval which is similar to the circle case.

\begin{example}
Let $(M, g) = ([0, 2\pi a], dt^2)$. We have
\begin{enumerate}
\item If $\lambda > 0$ and $a\sqrt{\tau}$ is an integer, then
\begin{eqnarray*}
W(M; -\tau g)_D & = & \set{C \sin (\sqrt{\tau}t) : C \in \Real} \\
W(M; -\tau g)_N & = & \set{C \cos(\sqrt{\tau}t) : C \in \Real}.
\end{eqnarray*}
\item If $\lambda = 0$, then
\begin{eqnarray*}
W(M; -\tau g)_D & = & \set{0} \\
W(M; -\tau g)_N & = & \set{C : C \in \Real}.
\end{eqnarray*}
\item Otherwise
\[
W(M; -\tau g) = W(M;-\tau g)_D = W(M; -\tau g)_N = \set{0}.
\]
\end{enumerate}
\end{example}

\medskip
\section{Warped Product Extensions}

In this section we create a fairly general class of examples using warped product extensions. The goal is to start with a base space $\left(B,g_{B},u\right)$ and  then construct $(M,g)$ as a warped product over $(B,g_{B})$ with fiber $\left(F,g_{F}\right)$ and metric given by
\[
g=g_{B}+u^{2}g_{F}.
\]
When $\partial B\neq\emptyset$ there are further conditions in order to obtain a smooth metric on $M$. The fiber has to be a round sphere which we can assume to be the unit sphere and $\nabla u$ a unit normal field to $\partial B\subset B$. There are further conditions on the higher derivatives of $u$ and we also need $\partial B\subset B$ to be totally geodesic. These conditions, however, are automatically satisfied as we assume that
\[
uq_{B}=\mathrm{Hess}_{B}u
\]
for some smooth symmetric tensor $q_{B}$ on $B$.

The warped product structure defines two distributions on $M$, the horizontal one given by $TB$ and the vertical one by $TF$. We denote these two distributions by $\mathcal{B}$ and $\mathcal{F}$ respectively. The projection onto $B$ is denoted by $\pi_{1}:M\rightarrow B$ and the projection onto $F$ by $\pi_{2}:M\rightarrow F$. We use $X,Y,\ldots$ and $U,V,\ldots$ to denote the horizontal and vertical vector fields respectively.

Next we need to define $q$ on $M$ as an extension of $q_{B}$ on $B$. We assume that $q$ preserves the horizontal and vertical distributions and that on the horizontal distribution $q\left(X,Y\right)=q_{B}\left(X,Y\right)$.

As $q$ preserves the horizontal and vertical distributions it follows that any $w\in W\left(M;q\right)$ has the property that its Hessian also preserves these distributions. Consequently the function $w$ has a special form.

\begin{lem}\label{lem:wsplitting}
If $M=B\times_uF$ and  $w:M\rightarrow\mathbb{R}$ satisfies
\[
(\mathrm{Hess}_{M}w)(X,U)=0
\]
for all $X\in TB$ and $U\in TF$, then
\[
w=\pi_{1}^{*}(z)+\pi_{1}^{*}(u)\cdot\pi_{2}^{*}(v)
\]
where $z:B\rightarrow\mathbb{R}$ and $v:F\rightarrow\mathbb{R}$.

Moreover, if
\[
\pi_{1}^{*}(z)+\pi_{1}^{*}(u)\cdot\pi_{2}^{*}(v)=0
\]
then $v$ must be constant and $z$ a multiple of $u$.
\end{lem}

\begin{proof}
The second fundamental form for a warped product is particularly simple: if $X$ is a vector field on $B$ and $U$ a vector field on $F$, then
\[
\nabla_{X}^{M}U=\nabla_{U}^{M}X=\frac{D_{X}u}{u}U.
\]
With that in mind we have
\begin{eqnarray*}
D_{X}\frac{1}{u}D_{U}w & = & -\frac{D_{X}u}{u^{2}}D_{U}w+\frac{1}{u}D_{X}D_{U}w\\
& = & \frac{1}{u}\left(-D_{\nabla_{X}^{M}U}w+D_{X}D_{U}w\right)\\
& = & \frac{1}{u}(\mathrm{Hess}_{M}w)(X,U)\\
& = & 0.
\end{eqnarray*}
Thus $D_{U}\frac{w}{u}$ is constant on $B$. This shows that if we restrict $\frac{w}{u}$ to the fibers $F_{i}=\left\{ b_{i}\right\} \times F$ over points $b_{1},b_{2}\in B$ then the difference
\[
\frac{w}{u}|_{F_{1}}-\frac{w}{u}|_{F_{2}}
\]
is constant. This shows the claim.

For the uniqueness statement just note that if
\[
\pi_{1}^{*}(z)=-\pi_{1}^{*}(u)\cdot\pi_{2}^{*}(v)
\]
then the right hand side defines a function on $B$ and thus $v$ must be constant.
\end{proof}

\begin{rem}
When $B$ has boundary and we insist that both $\pi_{1}^{*}(z)$ and $\pi_{1}^{*}(u)\cdot\pi_{2}^{*}(v)$ be smooth on $M$, then there are extra conditions. The function $\pi_{1}^{*}(u)\cdot\pi_{2}^{*}(v)$ is smooth at the singular set only if $v$ is odd $-v\left(y\right)=v\left(-y\right),\: y\in\mathbb{S}^{k}$. On the other hand $\pi_{1}^{*}(z)$ can only be smooth if $\nabla z$ is tangent to the boundary of $B$, i.e., it satisfies the Neumann boundary condition on $B$.
\end{rem}

Next we study how $W(M;q)$ relates to $u$ and the fiber $F$.
\begin{prop}\label{prop:WsplittingHWP}
Let $M=B \times_uF$ and assume that $uq_{B}=\mathrm{Hess}_{B}u$. Then $w\in W(M;q)$ if and only if there exist $z\in C^{\infty}(B)$ and $v\in C^{\infty}(F)$ such that
\begin{enumerate}
\item $w=\pi_{1}^{*}(z)+\pi_{1}^{*}(u)\cdot\pi_{2}^{*}(v)$,
\item $z\in W(B;q_{B})$, and
\item $\mathrm{Hess}_{F}v+v\left(-q|_{\mathcal{F}}+\left|\nabla u\right|_{B}^{2}g_{F}\right)=-\left(-\dfrac{z}{u}q|_{\mathcal{F}}+g_{B}(\nabla u,\nabla z)g_{F}\right)$.
\end{enumerate}
\end{prop}

\begin{proof}
From Lemma \ref{lem:wsplitting} we know that any function $w\in W(M;q)$ has the form
\[
w=\pi_{1}^{*}(z)+\pi_{1}^{*}(u)\cdot\pi_{2}^{*}(v).
\]

On the horizontal distribution we have
\begin{eqnarray*}
w q|_{\mathcal{B}} & = & zq_{B}+uvq_{B},\\
\left(\mathrm{Hess}_{M}w\right)|_{\mathcal{B}} & = & \mathrm{Hess}_{B}z+v\mathrm{Hess}_{B}u.
\end{eqnarray*}
Since $u\in W\left(B;q_{B}\right)$ we see that $\left(\mathrm{Hess}_{M}w\right)|_{\mathcal{B}}=wq|_{\mathcal{B}}$ if and only if $z\in W\left(B;q_{B}\right)$.

On the vertical distribution we have
\begin{eqnarray*}
w q|_{\mathcal{F}} & = & zq|_{\mathcal{F}}+uvq|_{\mathcal{F}}\\
\left(\mathrm{Hess}_{M}w\right)|_{\mathcal{F}} & = & u\mathrm{Hess}_{F}v+uv|\nabla u|_{B}^{2}g_{F}+ug_{B}(\nabla u,\nabla z)g_{F}.
\end{eqnarray*}
Thus $wq|_{\mathcal{F}}=\left(\mathrm{Hess}_{M}w\right)|_{\mathcal{F}}$ is equivalent to condition 3.
\end{proof}

Note that if $\dim W\left(B;q_{B}\right)=1$ then all $w\in W\left(M;q\right)$ are of the form $w=\pi_{1}^{*}(u)\cdot\pi_{2}^{*}(v)$. This motivates the following

\begin{cor}\label{cor:Psplitting}
Let $M=B \times_uF$ and assume that $uq_{B}=\mathrm{Hess}_{B}u$, $W\left(M;q\right)\neq0$, and that some nontrivial \textup{$w\in W\left(M;q\right)$ }\textup{\emph{is of the form $w=\pi_{1}^{*}(u)\cdot\pi_{2}^{*}(v)$, then}} there is a symmetric tensor $q_{F}$ on $F$ such that
\[
q_{F}=q|_{\mathcal{F}}-\left|\nabla u\right|_{B}^{2}g_{F}.
\]
\end{cor}

\begin{proof}
As we can write $w\in W\left(M;q\right)$ in the form $w=\pi_{1}^{*}(u)\cdot\pi_{2}^{*}(v)$ it follows from condition 3. in Proposition \ref{prop:WsplittingHWP} that
\[
\mathrm{Hess}_{F}v=v\left(q|_{\mathcal{F}}-\left|\nabla u\right|_{B}^{2}g_{F}\right).
\]
This implies that $q|_{\mathcal{F}}-\left|\nabla u\right|_{B}^{2}g_{F}$ can only depend on $F$ at points where $v\neq0$. Now $w$ and hence also $v$ can only vanish on a totally geodesic hypersurface so by continuity $q_{F}$ defines a symmetric tensor on $F$.
\end{proof}

\begin{rem}
As we shall see, the most important examples of such constructions always have the property that $q|_{\mathcal{F}}=-\kappa u^{2}g_{F}$ for some function $\kappa:M\rightarrow\mathbb{R}$. The previous corollary then shows that $\kappa u^{2}+\left|\nabla u\right|^{2}$ is constant on the horizontal leaves and thus defines a function on $F$. 
\end{rem}

\medskip
\section{The Warped Product Structure}

In this section we prove Theorem A, i.e., manifolds with $\dim W\left(M;q\right)>1$ are warped products, see Theorem \ref{thm:WPHWP}.

We start with an elementary lemma that shows how we construct Killing vector fields.
\begin{lem}\label{lem:Killing-Construction}
Let $v,w\in C^{\infty}\left(M\right)$ then $v\nabla w-w\nabla v$ is a Killing vector field if and only if $v\mathrm{Hess}w=w\mathrm{Hess}v$.
\end{lem}

\begin{proof}
We prove this by a simple direct calculation:
\begin{eqnarray*}
\nabla_{X}\left(v\nabla w-w\nabla v\right) & = & (D_{X}v)\nabla w-(D_{X}w)\nabla v+v\nabla_{X}\nabla w-w\nabla_{X}\nabla v\\
& = & g\left(\nabla v,X\right)\nabla w-g\left(\nabla w,X\right)\nabla v+v\nabla_{X}\nabla w-w\nabla_{X}\nabla v\\
& = & \left(\nabla w\wedge\nabla v\right)\left(X\right)+v\nabla_{X}\nabla w-w\nabla_{X}\nabla v.
\end{eqnarray*}
This shows that $v\nabla w-w\nabla v$ is a Killing vector field precisely when
\[
v\nabla_{X}\nabla w=w\nabla_{X}\nabla v
\]
which finishes the proof.
\end{proof}

From the lemma above there is a natural map from the exterior square of a subspace $W \subset W(M; q)$ to the Lie algebra of the isometry group of $(M, g)$:
\begin{eqnarray}
\iota : \wedge^2 W & \rightarrow & \mathfrak{iso}(M, g) \\ \nonumber
v\wedge w  & \mapsto & v \nabla w - w\nabla v.\label{eq:iota}
\end{eqnarray}
The map $\iota$ is injective by Proposition \ref{prop:injectionHWP}. In section 7 we shall see that for many interesting examples of $M$ there is a bilinear form $\bar{\mu}$ associated with $W$ that yields a Lie algebra structure on $\wedge^2 W \subset \mathfrak{so}(W, \bmu)$. In Proposition \ref{prop:liealg} we prove that this structure is compatible with the Lie algebra structure of $\mathfrak{iso}(M, g)$.

For the remainder of this section we fix a Riemannian manifold $\left(M,g\right)$ and a quadratic form $q$ on $M$. Furthermore, we select a subspace $W\subset W\left(M;q\right)$. Lemma \ref{lem:Killing-Construction} tells us that we get Killing vector fields from it when $\dim W > 1$.  For each such subspace we define
\[
W_{p}=\left\{ w\in W:w\left(p\right)=0\right\}
\]
A point $p$ is said to be \emph{regular} if the codimension of $W_{p}\subset W$ is one. Otherwise a point is called \emph{singular}. The set of singular points is denoted $S$.

\begin{prop}\label{prop:singularset}
The singular set $S$ is a totally geodesic submanifold of codimension $\dim W$.
\end{prop}

\begin{proof}
It follows by induction from Proposition \ref{prop:hypersurfaceHWP} that $S$ has the stated properties.
\end{proof}

At regular points $p\in M-S$ we define
\[
\mathcal{F}_{p}=\left\{ \nabla w:w\in W_{p}\right\}
\]
and let $\mathcal{B}$ be the orthogonal distribution on $M-S$. At a regular point there is a unique $u_{p}\in W$ with
\begin{eqnarray*}
u_{p}\left(p\right) & = & 1\\
\nabla u_{p}|_{p} & \perp & \mathcal{F}_{p}.
\end{eqnarray*}
This orthogonal distribution has the following properties.

\begin{prop}\label{prop:Btotallygeod}
Suppose $\dim W\geq1$ and let $k=\dim W-1$. Then $\mathcal{B}$ is integrable on the regular set $M-S$ and the leaves are totally geodesic of dimension $n-k$. Let $B_{p}$ be the leaf of the foliation $\mathcal{B}$ through $p\in M-S$, then $u_{p}$ is positive on $B_{p}$. Finally $q$ preserves the two distributions.
\end{prop}

\begin{proof}
Recall that $\mathcal{B}$ is the orthogonal distribution to $\mathcal{F}$ and
\[
\mathcal{F}_{p}=\left\{ \nabla w:w\in W_{p}\right\} .
\]
If two vector fields are perpendicular to the gradient of a function, then their Lie bracket is clearly also perpendicular to the gradient. This shows that $\mathcal{B}$ is integrable. Moreover the leaf through $p\in M-S$ is the connected component $B_{p}$ in
\[
\left\{ x\in M-S:w\left(x\right)=0\text{ for all }w\in W_{p}\right\}
\]
that contains $p.$ This is clearly a totally geodesic submanifold. If $u_{p}$ vanishes at $x\in B_{p},$ then $u_{p}\in W_{x}$ and consequently also lies in $W_{p}$, a contradiction.

Note that on $T_{p}M$ we have
\[
q\left(X,V\right)=g\left(\nabla_{X}\nabla u_{p},V\right).
\]
As $\nabla u_{p}$ is tangent to $B_{p}$ and $B_{p}$ is totally geodesic it follows that $q\left(X,V\right)=0$ if $X\in T_{p}B_{p}$ and $V\in\mathcal{F}_{p}$.
\end{proof}

\begin{rem}
Note that when $W=\left\{ 0\right\} $ we have $S=M.$ In the next case where $\dim W=1$ the regular set $M-S$ has two components. Each of these components is a leaf in the totally geodesic foliation $\mathcal{B}$.
\end{rem}
\medskip{}

\begin{rem}
Note that $B_{p}$ need not be complete even if $M$ is. It can however be completed by adding components of $S$ as boundary pieces. Thus the closure $\overline{B}_{p}$ is naturally a manifold with boundary when $S\neq\emptyset.$
\end{rem}

Next we investigate the $\left(\dim W-1\right)$-dimensional distribution $\mathcal{F}$ as well as its extension
\[
\widehat{\mathcal{F}}_{p}=\left\{ \nabla w|_{p}:w\in W\right\} .
\]

\begin{prop} \label{prop:Fintegrable}
Suppose $\dim W=k+1\geq2$. The distribution $\mathcal{F}$ on $M-S$ is integrable and is generated by a set of Killing vector fields on $M$ of dimension $\frac{1}{2}k(k+1)$. Moreover, for any vector fields $Z\in\mathcal{F}$ and $X\in TM$, we have
\[
\nabla_{X}Z\in\widehat{\mathcal{F}}.
\]
\end{prop}

\begin{proof}
For a fixed point $p\in M-S$, the space $W_{p}$ is spanned by the following functions
\[
v(p)w-w(p)v,\quad\mbox{ for }v,w\in W.
\]
It follows that the following vectors form a spanning set of the subspace $\mathcal{F}_{p}\subset T_{p}M$:
\[
v(p)\nabla w|_{p}-w(p)\nabla v|_{p},\quad\mbox{ for }v,w\in W.
\]
So we can write the distribution $\mathcal{F}$ as
\[
\mathcal{F}=\left\{ v\nabla w-w\nabla v:v,w\in W\right\} .
\]
Note that $\widehat{\mathcal{F}}$ might not be a distribution on $M-S$ as the dimension of $\widehat{\mathcal{F}}_{p}$ can be either $k+1$ or $k$. It agrees with $\mathcal{F}_{p}$ for those $p$ where its dimension is $k$. At the points where its dimension is $k+1$ the complementary subspace of $\mathcal{F}_{p}\subset\widehat{\mathcal{F}}_{p}$ is one-dimensional and spanned by $\nabla u_{p}|_{p}$.

Using $\mathcal{F}=\left\{ v\nabla w-w\nabla v:v,w\in W\right\} $ we see that when $X\in\mathcal{B}$, Proposition \ref{prop:Btotallygeod} implies
\[
g\left(\left[\nabla v,\nabla w\right],X\right)=-q\left(v\nabla w-w\nabla v,X\right)=0
\]
In particular, both $\widehat{\mathcal{F}}$ and $\mathcal{F}$ are integrable where they are distributions. Moreover, we know from Lemma \ref{lem:Killing-Construction} that $\mathcal{F}$ is spanned by Killing vector fields.

Finally we calculate the dimension of this set of Killing vector fields on $M$. First note that it can't exceed $\frac{1}{2}k\left(k+1\right)$ as the fields are all tangent to a $k$-dimensional distribution. Next note that at $p\in M-S$ we have two types of Killing vector fields
\[
v\nabla w-w\nabla v,\, v,w\in W_{p}
\]
and
\[
u_{p}\nabla w-w\nabla u_{p},\, w\in W_{p}.
\]
The first type of Killing vector field vanishes at $p$ and has covariant derivative $\nabla w|_{p}\wedge\nabla v|_{p}$ which defines a skew symmetric transformation that leaves $\mathcal{F}_{p}$ invariant. Moreover, as the skew symmetric transformations on $\mathcal{F}_{p}$ are generated by such transformations these Killing vector fields generate a subspace of dimension at least
\[
\frac{1}{2}k\left(k-1\right).
\]
The second type of Killing vector field has value $\nabla w|_{p}$ at $p$. Thus these Killing vector fields will generate a complementary subspace of dimension at least $k$. This shows that the Killing vector fields $\left\{ v\nabla w-w\nabla v:v,w\in W\right\}$ generate a space of Killing vector fields of dimension at least $\frac{1}{2}k\left(k+1\right)$.
\end{proof}

\begin{cor}\label{cor:wedgeWliealg}
The image $\iota(\wedge^2 W) \subset \mathfrak{iso}(M, g)$ is a Lie subalgebra.
\end{cor}
\begin{proof}
We already know that $\iota(\wedge^2 W)$ has dimension $\frac{1}{2}k(k+1)$. If it is not a Lie subalgebra of $\mathfrak{iso}(M, g)$, then it would generate a strictly larger space $V$ which is a Lie subalgebra. Since $\mathcal{F}$ is integrable, we have
\[
\iota(\wedge^2 W)\subsetneqq V \subset \mathcal{F}.
\]
This contradicts the fact that $\mathcal{F}$ has dimension $\frac{1}{2}k(k+1)$.
\end{proof}

We can now prove our Theorem A from the introduction.
\begin{thm}\label{thm:WPHWP}
Let $\left(M^{n},g\right)$ be a complete simply connected Riemannian manifold with a symmetric tensor $q$ and $W$ a subspace of  $W(M;q)$. If $\dim W=k+1\geq2$, then
\[
M=B\times_{u}F
\]
where $u$ vanishes on the boundary of $B$. Moreover, $F$ is either the k-dimensional unit sphere $\mathbb{S}^{k}\left(1\right)\subset\mathbb{R}^{k+1}$, k-dimensional Euclidean space $\mathbb{R}^{k}$, or the k-dimensional hyperbolic space $H^{k}$. In the first two cases $k\geq1$ while in the last $k>1$.
\end{thm}

\begin{proof}
Proposition \ref{prop:Fintegrable} shows that the set of Killing vector fields on $M$ that are tangent to the foliation $\mathcal{F}$ is a subalgebra of the space of all Killing vector fields on $M$ of dimension $\frac{1}{2}k(k+1)$. As $M$ is complete this means that there'll be a corresponding connected subgroup $G\subset\mathrm{Iso}\left(M,g\right)$. First observe that as the Killing vector fields $v\nabla w-w\nabla v$ vanish on $S$, the group $G$ fixes $S$. Next note that $G$ forces the leaves of the foliation $\mathcal{F}$ to be maximally symmetric. In particular, they are complete connected space forms, which are either simply connected or possibly circles or real projective spaces see e.g., \cite[page 190]{Petersen}. From what we show below it'll be clear that the case of real projective spaces will not occur here as $M$ is simply connected.

We wish to show that the quotient map $\pi_1:M \rightarrow M/G$ is a Riemannian submersion on $M -S$.  When there is no singular set, this follows from  \cite[Theorem A]{BH}.   In fact, due to the group action $G$ the proof of \cite[Theorem A]{BH} is somewhat simpler in our case and can be adapted to work in case $S \neq \emptyset$ and $k>1$.  That is, the case where  $M-S$ is connected and simply connected (see also \cite[p. 203]{O'Neill} for a similar construction in the context of covering spaces).

First note that, when at least one fiber $F_{p}$ is compact,  $G$ itself is compact and so the action is proper. In particular, if $S\neq\emptyset$, then the fibers $F_{p}$ for $p$ near $x\in S$ can be identified with the space of unit normal vectors to $x\in S$ and so the fibers are compact.

For each $p\in M-S$ there is a neighborhood $U_{p}$ and a uniquely defined Riemannian submersion $U_{p}\rightarrow B_{p}$ which projects along the leaves of $\mathcal{F}$. Next note that any two vertical leaves can be connected by a horizontal geodesic in $M-S$. This shows that $G$ acts transitively on the leaves $B_{p}$, $p\in M-S$. Now fix a specific horizontal leaf $B$. By using elements in $G$ we can then construct Riemannian submersions $f_{p}:U_{p}\rightarrow B$ with the properties that: If $U_{p_{1}}\cap U_{p_{2}}\neq\emptyset$, then there exits $h\in G$ such that $h\left(B\right)=B$ and $h\circ f_{p_{1}}=f_{p_{2}}$. Since $M-S$ is connected and simply connected a standard monodromy argument then shows that we obtain a global Riemannian submersion $f:M-S\rightarrow B$. Moreover, $\bar{B}=M/G$ so the natural projection $\pi_{1}:M\rightarrow M/G$ is a Riemannian submersion when restricted to $M-S$.

This leaves us with the situation where $k=1$ and $S\neq\emptyset$. In particular, all fibers are circles. In this case $G$ is Abelian. We start by observing in general that if some $h\in G$ fixes all points in a fiber $F_{p}$ and $S\neq\emptyset$, then $h$ acts trivially. Let $x\in S$ be the closest point to $p$. Then $h$ must fix the unique shortest geodesic from $x$ to $p$ in $B_{p}$. Note that it is unique as it is normal to $S$ in $\bar{B}_{p}$. Next observe that we can move this geodesic by isometries from $G$ to get minimal connections from $x$ to all other points in the orbit $F_{p}$. Since $h$ fixes all of $F_{p}$ we see that $h$ not only fixes $S$ but also all normal directions $\nu_{x}S$. Thus $h$ acts trivially. In case $G$ is Abelian this implies that all principal isotropy groups are trivial. In particular, $\pi_{1}:M\rightarrow M/G$ is a Riemannian submersion when restricted to $M-S$.

In all cases we now have that  the quotient map $\pi_1:M \rightarrow M/G$ is a Riemannian submersion on $M -S$.  Since $G\subset\mathrm{Iso}\left(M,g\right)$ the leaves of $\mathcal{F}$ have the property that their second fundamental forms are also invariant under $G$. This implies that the leaves are totally umbilic with a mean curvature vector that is invariant under the group action. As the orthogonal foliation is totally geodesic it follows that the mean curvature vector is basic. It then follows from \cite[Chapter 9.J]{Besse} that these two foliations yield a local warped product structure on $M - S$.

Since $G$ fixes $S$, to obtain a global warped product structure we need only show that $B_{p}\cap F_{p}=\left\{ p\right\}$ on $M -S$.  When there is no singular set we can again  appeal to   \cite[Theorem A]{BH}which says that in this case $M$ is diffeomorphic to $B_{p}\times F_{p}$.

When $S \neq \emptyset$ note that the quotient map $\pi_1:M \rightarrow M/G$ forces $M/G$ to be a Riemannian manifold with totally geodesic boundary $S$. In particular, $M/G$ is homotopy equivalent to its interior. In this situation we know initially only that $\pi_{1}$ is a Riemannian covering map when restricted to horizontal leaves. However, let $\gamma:\left[0,1\right]\rightarrow\mathrm{int}\left(M/G\right)$ be a loop and consider a horizontal lift $\bar{\gamma}:\left[0,1\right]\rightarrow M-S$. As $M$ is simply connected $\bar{\gamma}$ is homotopic to a path in the fiber
\[
\pi_{1}^{-1}\left(\gamma\left(0\right)\right)=\pi_{1}^{-1}\left(\gamma\left(1\right)\right)
\]
through a homotopy that keeps the endpoints fixed. This in turn shows that $\gamma$ is homotopic to a point in $M/G$. Thus $\mathrm{int}\left(M/G\right)$ is simply connected and we see that $\pi_{1}$ is an isometry when restricted to horizontal leaves.   In particular,  for all $p\in M-S$ we have $B_p \cap F_p = \left \{ p \right \}$.
\end{proof}

\begin{cor}\label{cor:ridigidity}
When $\mathcal{F}=\widehat{\mathcal{F}}$, i.e., the foliation $\mathcal{F}$ is totally geodesic, the manifold $M$ is isometric to a product.
\end{cor}

\medskip
\section{Properties of the Quadratic Form}

Assume below that we have a complete simply connected Riemannian $n$-manifold with $\dim W\left(M;q\right)\geq2$ and a fixed subspace $W\subset\dim W\left(M;q\right)$ with $\dim W=k+1$ and $k\geq1$. Theorem \ref{thm:WPHWP} then tells us that
\[
\left(M,g\right)=\left(B\times F,g_{B}+u^{2}g_{F}\right)
\]
for some function $u:B\rightarrow[0,\infty)$ that vanishes only on $\partial B$. In this section we give the details of how to show that the base is a base space and the fiber a fiber space.

Note that we shall not distinguish between fields on $B$ and their corresponding horizontal lifts to $M$. However, we will be careful with notation in regards to derivatives of such fields. We'll use $A_1$, $A_2$ as vector fields on $M$, $X,Y$ as horizontal fields and $U,V$ as vertical fields. Also we shall for convenience use $u$ for its pullback to $M$.

The vertical isometries from $G$ act as isometries on $M$ and so there is a function $\rho:B\rightarrow\mathbb{R}$ such that
\[
\mathrm{Ric}_{M}\left(V\right)=\rho V,\: V\in\mathcal{F}.
\]
From \cite[Chapter 9]{Besse} we obtain the following facts for warped products: the vertical Ricci curvature $\rho$ is related to the Einstein constant $\rho^{F}$ for $F$ by
\[
\rho u^{2}+u\Delta_{B}u+(k-1)|\nabla u|^{2}=\rho^{F}.
\]
The horizontal Ricci curvatures satisfy
\[
\mathrm{Ric}_{M}(X,Y)=\mathrm{Ric}_{B}(X,Y)-\frac{k}{u}(\mathrm{Hess}_{B}u)(X,Y).
\]
The extrinsic geometry of the leaves of $\mathcal{F}$ are governed by
\begin{eqnarray*}
g\left(\nabla_{V}V,X\right) & = & -\frac{1}{u}g\left(X,\nabla u\right)g\left(V,V\right).
\end{eqnarray*}
In particular,
\[
\nabla_{V}\nabla u=\frac{\left|\nabla u\right|^{2}}{u}V.
\]

The goal here is to show that $q$ depends only on $\mathrm{tr}Q$, where $q\left(A_1,A_2\right)=g\left(Q(A_1),A_2\right)$, and the Ricci curvatures of $B$ and $M$.

We start by relating the elements in $W$ to the warping function $u$.

\begin{lem} \label{lem:u-structure}
For any $w\in W$ we have
\[
g\left(\nabla w,\nabla u\right)=\frac{\left|\nabla u\right|^{2}}{u}w.
\]
Moreover on $B_{p}$, the horizontal leaf through $p$, we have $u_{p}=\frac{u}{u\left(p\right)}$.
\end{lem}

\begin{proof}
First note that $\nabla u$ is basic and invariant under the group action $G$, and thus commutes with the Killing vector fields $v\nabla w-w\nabla v$. This shows
\begin{eqnarray*}
\nabla_{\nabla u}\left(v\nabla w-w\nabla v\right) & = & \nabla_{v\nabla w-w\nabla v}\nabla u\\
& = & \frac{\left|\nabla u\right|^{2}}{u}\left(v\nabla w-w\nabla v\right)
\end{eqnarray*}
but the left hand side is also
\begin{eqnarray*}
\nabla_{\nabla u}\left(v\nabla w-w\nabla v\right) & = & g\left(\nabla v,\nabla u\right)\nabla w-g\left(\nabla w,\nabla u\right)\nabla v.
\end{eqnarray*}
As long as $\nabla v$ and $\nabla w$ are linearly independent this shows
\[
g\left(\nabla w,\nabla u\right)=\frac{\left|\nabla u\right|^{2}}{u}w.
\]
As $v,w\in W$ are arbitrary we have shown that this holds for all $w\in W$.

Next we claim that $\nabla u_{p}$ stays tangent to $B_{p}$. Let $w\in W_{p}$ then $w$ vanishes on $B_{p}$. So for $X\in TB_{p}$ we have
\begin{eqnarray*}
D_{X}g\left(\nabla u_{p},\nabla w\right) & = & \mathrm{Hess}u_{p}\left(X,\nabla w\right)+\mathrm{Hess}w\left(X,\nabla u_{p}\right)\\
& = & u_{p}q\left(X,\nabla w\right)+wq\left(\nabla u_{p},X\right)\\
& = & 0.
\end{eqnarray*}
As $g\left(\nabla u_{p},\nabla w\right)=0$ at $p$, this shows that $g\left(\nabla u_{p},\nabla w\right)=0$ on all of $B_{p}$. Next recall from Proposition \ref{prop:Fintegrable} that
\[
\nabla_{V}V\in\widehat{\mathcal{F}}.
\]
In particular, it follows that $\nabla u\in\widehat{\mathcal{F}}\cap\mathcal{B}$. We clearly also have $\nabla u_{p}\in\widehat{\mathcal{F}}\cap\mathcal{B}$ so it follows that
\begin{eqnarray*}
\nabla u_{p} & = & g\left(\nabla u,\nabla u_{p}\right)\frac{\nabla u}{\left|\nabla u\right|^{2}}\\
& = & \frac{1}{u}u_{p}\nabla u.
\end{eqnarray*}
From which we get the last claim.
\end{proof}

This lemma allows us to completely determine the horizontal structure of $q$.
\begin{thm}\label{thm:horizontalP}
On $\mathcal{B}$ we have
\[
q|{}_{\mathcal{B}}=\frac{1}{k}\left(\mathrm{Ric}_{B}-\mathrm{Ric}_{M}\right)=\frac{1}{u}\mathrm{Hess}_{B}u.
\]
On the base space $B$, the quadratic form is given by
\[
q_{B}=\frac{1}{u}\mathrm{Hess}_{B}u.
\]
\end{thm}

\begin{proof}
We calculate on $B_{p}$ and use the linear operator $Q$ corresponding to $q$
\begin{eqnarray*}
Q\left(X\right) & = & \frac{1}{u_{p}}\nabla_{X}\nabla u_{p}\\
& = & \frac{1}{u}\nabla_{X}\nabla u\\
& = & \frac{1}{k}\left(\mathrm{Ric}_{B}-\mathrm{Ric}_{M}\right)\left(X\right).
\end{eqnarray*}
The second statement follows as $(B, g_B)$ is totally geodesic in $M$.
\end{proof}

Next we turn our attention to the vertical structure of $q$.
\begin{thm}\label{thm:verticalP}
Restricting $q$ to the vertical fibers we have
\[
q|_{\mathcal{F}}=\left(\rho+\mathrm{tr}Q\right)g|_{\mathcal{F}}=\left(\rho+\mathrm{tr}Q\right)u^{2}g{}_{F}.
\]
\end{thm}

\begin{rem}
Note that we cannot expect any more information given what happens when $\dim M=1$ as $F=M$ in that case.
\end{rem}

\begin{proof}
Start with $w\in W$, i.e.,
\[
\nabla\nabla w=wQ.
\]
The Weitzenb\"{o}ck formula for a gradient field $\nabla w$ states
\[
\mathrm{div}\nabla\nabla w=\nabla\Delta w+\mathrm{Ric}\left(\nabla w\right)
\]
which for our specific field reduces to
\[
\mathrm{div}\left(wQ\right)=\nabla\left(w\mathrm{tr}Q\right)+\mathrm{Ric}\left(\nabla w\right).
\]
This implies
\[
Q\left(\nabla w\right)+w\mathrm{div}\left(Q\right)=\mathrm{tr}Q\nabla w+w\nabla\left(\mathrm{tr}Q\right)+\mathrm{Ric}\left(\nabla w\right).
\]
So
\[
Q\left(\nabla w\right)=\mathrm{Ric}\left(\nabla w\right)+\mathrm{tr}Q\nabla w+w\left(-\mathrm{div}\left(Q\right)+\nabla\left(\mathrm{tr}Q\right)\right).
\]
This tells us that $Q$ is essentially determined by $\mathrm{Ric}$, $\mathrm{tr}Q$ and $\mathrm{div}\left(Q\right)$ on $\widehat{\mathcal{F}}$. On $\mathcal{F}$ we can be more specific. Let $p\in M$ and $w\in W_{p}$ then
\[
Q\left(\left(\nabla w\right)|_{p}\right)=\mathrm{Ric}\left(\left(\nabla w\right)|_{p}\right)+\left(\mathrm{tr}Q\right)\left(\nabla w\right)|_{p}
\]
showing that
\[
Q|_{\mathcal{F}}=\left(\rho+\mathrm{tr}Q\right)I|_{\mathcal{F}}.
\]
\end{proof}

Finally we note that when $k>1$ then $q$ is completely determined by the vertical Ricci curvatures and the warping function.
\begin{cor}\label{cor:trP}
When $k=1$ we have
\[
\mathrm{tr}\left(Q{}_{B}\right)=\frac{\Delta_{B}u}{u}=-\rho,
\]
while if $k>1$
\[
\mathrm{tr}Q=-\frac{1}{k-1}\left(k\rho+\mathrm{tr}\left(Q{}_{B}\right)\right).
\]
In particular, $q$ is invariant under the action of $G$ if $k>1$.
\end{cor}

\begin{proof}
Our formulas for $q|_{\mathcal{B}}$ and $q|_{\mathcal{F}}$ imply that
\[
\mathrm{tr}Q=k\left(\rho+\mathrm{tr}Q\right)+\mathrm{tr}\left(Q{}_{B}\right)
\]
and by definition
\[
\mathrm{tr}\left(Q_{B}\right)=\frac{\Delta_{B}u}{u}.
\]
Both statements follow immediately from this.

For the last statement note that both $\rho$ and $\dfrac{\Delta_{B}u}{u}$ are invariant under $G$.
\end{proof}

\medskip
\section{The Structure of $W$}

For a given warped product structure coming from a specific subspace
$W\subset W\left(M;q\right)$ define
\[
\kappa=-\rho-\mathrm{tr}Q
\]
and
\begin{eqnarray}
\bar{\mu}\left(u\right) & = & \kappa u^{2}+\left|\nabla u\right|^{2}, \label{eq:barmuu}\\
\bar{\mu}\left(u,z\right) & = & \kappa uz+g\left(\nabla u,\nabla z\right). \nonumber
\end{eqnarray}

In this section we prove Theorem B for the subspace $W$. The argument is split into two cases.   In Theorem \ref{thm:W-S-notempty}  we prove the result when the singular set $S$ is nonempty and in Theorem \ref{thm:W} we address the case were $S = \emptyset$. In both cases we will see that the characteristic function $\tau$ of the fiber space $F$ is equal to $\bmu(u)$.

We first simplify Proposition \ref{prop:WsplittingHWP} by using Theorem \ref{thm:verticalP}.

\begin{prop}\label{prop:WsplittingHWP-1}
Let $M = B \times_u F$ and assume that $uq_{B}=\mathrm{Hess}_{B}u$. Then $w\in W(M;q)$ if and only if there exist $z\in C^{\infty}(B)$ and $v\in C^{\infty}(F)$ such that
\begin{enumerate}
\item $w=\pi_{1}^{*}(z)+\pi_{1}^{*}(u)\cdot\pi_{2}^{*}(v)$,

\item $z\in W(B;q_{B})$, and

\item $\mathrm{Hess}_{F}v+v\bar{\mu}\left(u\right)g_{F}=-\bar{\mu}\left(u,z\right)g_{F}$.
\end{enumerate}
\end{prop}

We note that, if $W(B; q_B)$ is spanned by $u$, then from Proposition \ref{prop:WsplittingHWP-1} above the conclusion to Theorem B holds.    However,  this is not always the case  as the following example shows.

\begin{example}\label{Exa:B=R}
There are examples such that $\dim W\left(M;q\right)=\dim M=k+1$, $M = B \times_u F^k$ and $\dim W(B; q_B) = 2$.  Let $B = (\Real, dx^2)$ be the real line. Select $u:\mathbb{R}\rightarrow\left(0,\infty\right)$ and define $q_{B}=\frac{u^{\prime\prime}}{u}dx^{2}$ where we use $^{\prime}$ for derivatives on the base. In this case we have $\dim W\left(B;\frac{u^{\prime\prime}}{u}dx^{2}\right)=2$. Next choose a simply connected $k$-dimensional fiber space $\left(F,g_{F},-\tau g_{F}\right)$, where $\tau$ is a constant when $k>1$ or a merely a function on $F=\mathbb{R}$. The warped product
\[
\left(M^{k+1},g,q\right)=\left(\mathbb{R}\times F,dx^{2}+u^{2}g_{F},\frac{u^{\prime\prime}}{u}dx^{2}+\left(\left(u^{\prime}\right)^{2}-\tau\right)g_{F}\right)
\]
has the property that
\[
W\left(M;q\right)\supset\left\{ \pi_{1}^{*}(u)\cdot\pi_{2}^{*}(v):v\in W\left(F;-\tau g_{F}\right)\right\}
\]
and therefore has dimension $k+1$ or $k+2$. In the latter case the metric $dx^{2}+u^{2}g_{F}$ is forced to have constant curvature by Theorem A. In this case the precise condition for $\dim W\left(M;q\right)=k+2$ is
\[
\frac{\tau-\left(u^{\prime}\right)^{2}}{u^{2}}=\kappa=\kappa_{B}=-\frac{u^{\prime\prime}}{u}
\]
or
\[
\frac{\tau}{u^{2}}=-\left(\frac{u^{\prime}}{u}\right)^{\prime}.
\]
Note, in particular, when $F = \Real$ we can choose a non-constant $\tau$ on $F$ such that the above identity can never happen. So, as long as $u$, $\left(F,g_{F}\right)$, and $\tau$ are selected in such a way that the total space $M$ doesn't have constant curvature we obtain examples satisfying the desired conditions. 
\end{example}

Now we prove Theorem B in  the case where there is a singular set.

\begin{thm}\label{thm:W-S-notempty}
Let $\left(M,g\right)$ be complete and simply connected and $W\subset W\left(M;q\right)$ a subspace of dimension $k+1\geq2$. When $S\neq\emptyset$, then we have
\[
W=\left\{ \pi_{1}^{*}(u)\cdot\pi_{2}^{*}(v):v\in W\left(F;-\bar{\mu}\left(u\right)g_{F}\right)\right\}
\]
and $\bar{\mu}\left(u\right)$ is constant on $M$.
\end{thm}

\begin{proof}
We start by showing that any $w\in W$ is of the form $w=\pi_{1}^{*}(u)\cdot\pi_{2}^{*}(v)$. We know from Lemma \ref{lem:wsplitting} that
\[
w=\pi_{1}^{*}(z)+\pi_{1}^{*}(u)\cdot\pi_{2}^{*}(v).
\]
So the goal is simply to show that $z$ must be a multiple of $u$. Recall from Proposition \ref{prop:WsplittingHWP} that $z\in W\left(B;q_{B}\right)$.

When $S\neq\emptyset$ we know that $w|_{S}=0$ so it immediately follows that $z|_{\partial B}=0$. Then Proposition \ref{prop:injectionHWP} shows that $z=Cu$ for some constant $C$. We can now use Theorem \ref{thm:verticalP} and argue as in Corollary \ref{cor:Psplitting} that
\begin{eqnarray*}
q_{F} & = & q|_{\mathcal{F}}-\left|\nabla u\right|_{B}^{2}g_{F}\\
& = & -\left(\kappa u^{2}+\left|\nabla u\right|^{2}\right)g_{F}\\
& = & -\bar{\mu}\left(u\right)g_{F}
\end{eqnarray*}
defines a quadratic form on $F$. In particular $\bar{\mu}\left(u\right)$ is a function on $F$. This gives us the desired structure
\[
W=\left\{ \pi_{1}^{*}(u)\cdot\pi_{2}^{*}(v):v\in W\left(F;-\bar{\mu}\left(u\right)g_{F}\right)\right\} .
\]
Finally we show that $\bmu(u)$ is constant on $M$. We can think of $p\in M-S$ as a pair of points $p=\left(x,y\right)\in\mathrm{int}B\times F$. Thus $\bar{\mu}\left(u\right)\left(x,y\right)$ is constant in $x$ for a fixed $y$. Letting $x\rightarrow x_{0}\in\partial B=S$ and using that $\kappa=-\rho-\mathrm{tr}Q$ is continuous on $M$ then shows that
\[
\bar{\mu}\left(u\right)\left(x,y\right)=\kappa\left(x_{0}\right)u^{2}\left(x_{0}\right)+\left|\nabla u\right|^{2}|_{x_{0}}=\left|\nabla u\right|^{2}|_{x_{0}}.
\]
Here the right hand side is clearly independent of $y$ and so the left hand side must be as well. This shows that $\bar{\mu}\left(u\right)$ is constant on $M$.
\end{proof}

When there is no singular set we have to work a little harder to prove the same result.

\begin{thm}\label{thm:W}
Let $\left(M,g\right)$ be complete and simply connected. If $W\subset W\left(M;q\right)$ has dimension $k+1\geq2$ and $S=\emptyset$, then $\bar{\mu}\left(u\right)$ is a function on $F$ and a constant when $k>1$. Moreover,
\[
W=\left\{ \pi_{1}^{*}(u)\cdot\pi_{2}^{*}(v):v\in W\left(F;-\bar{\mu}\left(u\right)g_{F}\right)\right\} .
\]
\end{thm}

\begin{proof}
We start by showing that $\bar{\mu}\left(u\right)$ is constant when $k>1$. The vertical part of the Ricci curvatures for a warped product implies in our case that
\[
u^{2}\rho+\left(u\Delta_{B}u+(k-1)|\nabla u|^{2}\right)=\rho_{F}
\]
which can be reduced to
\[
\left(\rho+\mathrm{tr}\left(Q{}_{B}\right)\right)u^{2}+\left(k-1\right)\left|\nabla u\right|^{2}=\rho_{F}.
\]
When $k>1$ the relationship developed in Corollary \ref{cor:trP} implies
\[
\kappa=-\rho-\mathrm{tr}Q=\frac{\rho+\mathrm{tr}\left(Q{}_{B}\right)}{k-1}.
\]
Thus
\[
\bar{\mu}\left(u\right)=\kappa u^{2}+\left|\nabla u\right|^{2}=\frac{\rho_{F}}{k-1}
\]
is constant.

We know from Lemma \ref{lem:wsplitting} that
\begin{equation}
w=\pi_{1}^{*}(z)+\pi_{1}^{*}(u)\cdot\pi_{2}^{*}(v)=z+u\cdot v\label{eq:wsplitting}
\end{equation}
If we take the gradient of this equation at some point $p=\left(x,y\right)\in B\times F$ then
\[
\nabla w=\left(\nabla z\right)|_{x}+v\left(y\right)\left(\nabla u\right)|_{x}+u\left(x\right)\left(\nabla v\right)|_{y}.
\]
In this decomposition
\[
\left(\nabla v\right)|_{y}\in\mathcal{F}
\]
and
\[
\left(\nabla z\right)|_{x}+v\left(y\right)\left(\nabla u\right)|_{x}\in\mathcal{B}\cap\widehat{\mathcal{F}}.
\]
Thus it follows that
\[
\left(\nabla z\right)|_{x}\in\mathcal{B}\cap\widehat{\mathcal{F}}
\]
which in turn implies that $\left(\nabla z\right)|_{x}$ and $\left(\nabla u\right)|_{x}$ are linearly dependent for all $b\in B$.

Let $W_{B}\subset W\left(B;q_{B}\right)$ be the subspace spanned by $u$ and all $z$ that appear in equation (\ref{eq:wsplitting}). As these functions all have proportional gradients the foliation $\widehat{\mathcal{F}}_{B}$ on $B$ defined by $W_{B}$ has dimension at most 1 and consequently $\dim W_{B}\leq2$.

If $\dim W_{B}=1$, then $W_{B}=\mathrm{span}\left\{ u\right\} $ and so we always have:
\[
w=\pi_{1}^{*}(u)\cdot\pi_{2}^{*}(v).
\]
We can then use Proposition \ref{prop:WsplittingHWP-1} and Corollary \ref{cor:Psplitting} to see that
\begin{eqnarray*}
q_{F} & = & q|_{\mathcal{F}}-\left|\nabla u\right|_{B}^{2}g_{F}\\
& = & -\left(\kappa u^{2}+\left|\nabla u\right|^{2}\right)g_{F}\\
& = & -\bar{\mu}\left(u\right)g_{F}
\end{eqnarray*}
defines a quadratic form on $F$. In particular, $\bar{\mu}\left(u\right)$ is constant on the horizontal leaves and $v\in W\left(F;-\bar{\mu}\left(u\right)g_{F}\right)$. Moreover, when $k=1$ we clearly have that $\dim W\left(F;-\bar{\mu}\left(u\right)g_{F}\right)=k+1$ as $F = \Real$, while if $k>1$ $\bar{\mu}\left(u\right)$ is constant so Lemma \ref{lem:ReverseObata} implies that $\dim W\left(F;-\bar{\mu}\left(u\right)g_{F}\right)=k+1$. Thus
\[
W=\left\{ \pi_{1}^{*}(u)\cdot\pi_{2}^{*}(v):v\in W\left(F;-\bar{\mu}\left(u\right)g_{F}\right)\right\} .
\]

If $\dim W_{B}=2$, then Corollary \ref{cor:ridigidity} applied to the space $W_{B} \subset W(B; q_B)$ shows that
\[
\left(B,g_{B}\right)=\left(H\times\mathbb{R},g_{H}+dt^{2}\right).
\]
Moreover the functions in $W_{B}$ are constant on $H\times\left\{ t_{0}\right\} $ for all $t_{0}\in\mathbb{R}$. This shows that $u=u\left(t\right)$ and
\[
W_{B}=\left\{ z\in C^{\infty}\left(\mathbb{R}\right):z^{\prime\prime}=z\frac{u^{\prime\prime}}{u}\right\} .
\]
The functions without a $z$ component
\[
W^{F}=\left\{ w\in W:w=\pi_{1}^{*}(u)\cdot\pi_{2}^{*}(v)\right\}
\]
will then form a subspace of dimension $k$. In particular, it is nonempty and thus Proposition \ref{prop:WsplittingHWP-1} shows that
\[
\mathrm{Hess}_{F}v=-v\bar{\mu}\left(u\right)g_{F}
\]
for some $v$. This shows via Corollary \ref{cor:Psplitting} that $\bar{\mu}\left(u\right)$ defines a function on $F$ and that
\begin{eqnarray*}
q_{F} & = & q|_{\mathcal{F}}-\left|\nabla u\right|_{B}^{2}g_{F}\\
& = & -\left(\kappa u^{2}+\left|\nabla u\right|^{2}\right)g_{F}\\
& = & -\bar{\mu}\left(u\right)g_{F}
\end{eqnarray*}
defines a quadratic form on $F$. We can then argue as before that $\dim W\left(F;-\bar{\mu}\left(u\right)g_{F}\right)=k+1$.

To reach a contradiction in this case select $w_{1}=z+uv_{1}\in W$ and $w_{2}=uv_{2}\in W^{F}-\left\{ 0\right\} $. Then
\begin{eqnarray*}
TF & \ni & w_{1}\nabla w_{2}-w_{2}\nabla w_{1}\\
& = & \left(z+uv_{1}\right)\left(u\nabla v_{2}+v_{2}\nabla u\right)-uv_{2}\left(\nabla z+u\nabla v_{1}+v_{1}\nabla u\right)\\
& = & v_{2}\left(z\nabla u-u\nabla z\right)+u^{2}\left(v_{1}\nabla v_{2}-v_{2}\nabla v_{1}\right)+zu\nabla v_{2}.
\end{eqnarray*}
Since $\nabla v_{1},\nabla v_{2}\in TF$ it follows that $z\nabla u-u\nabla z=0$ as $v_{2}$ is non-trivial. But this shows that $z$ is a multiple of $u$ contradicting that $\dim W_{B}=2$.
\end{proof}

Finally we show that we cannot expect $\bar{\mu}\left(u\right)$ to be constant unless we are in the situations covered by Theorems \ref{thm:W-S-notempty} and \ref{thm:W}.

\begin{example}\label{Exa:F=R}
Since $M$ is assumed to be simply connected, Theorems \ref{thm:W-S-notempty} and \ref{thm:W} show that the only case where $\bar{\mu}\left(u\right)$ might not be constant is when $F=\mathbb{R}$. We can construct examples of this type as follows. Fix a base manifold $\left(B,g_{B}\right)$ with a positive function $u:B\rightarrow\left(0,\infty\right)$ such that $\dim W\left(B;\frac{1}{u}\mathrm{Hess}u\right)=1$. Let
\[
M=B\times_{u}\mathbb{R},\: g=g_{B}+u^{2}dt^{2}
\]
and define
\[
q=\frac{1}{u}\mathrm{Hess}_{B}u+\left(\left|\nabla u\right|^{2}-\tau\right)dt^{2}
\]
where $\tau:\mathbb{R}\rightarrow\mathbb{R}$ is any smooth function. This gives us a complete collection of examples where $F=\mathbb{R}$ and $\dim W\left(M;q\right)=2$.
\end{example}

\medskip
\section{Invariant Groups and non-simply connected manifolds}

In this section we obtain a complete classification of when a warped product splitting holds for non-simply connected spaces with $\dim W(M;q) >1$, see Proposition \ref{prop:nonsimplyconn}. We do this through the study of isometries $h \in \mathrm{Iso}(M,g)$ which preserve the quadratic form $q$
\begin{eqnarray*}
h^{*}q=q.
\end{eqnarray*}

We start by extending the definition of $\bar{\mu}$ in (\ref{eq:barmuu}) to all of $W$:
\[
\bmu(w) = \kappa w^2 + \abs{\nabla w}^2.
\]
When $\bmu(u)$ is constant on $M$, so is $\bmu(w)$ for any $w \in W$, see Proposition \ref{prop:mu-construction}. In the following we show that $W\left(M;q\right)$ and $\bar{\mu}$ are preserved by such isometries.

\begin{prop}
If $h\in\mathrm{Iso}\left(M\right)$ and $h^{*}q=q$, then $h^{*}:W\left(M;q\right)\rightarrow W\left(M;q\right)$ preserves the characteristic constant/function $\bar{\mu}$.
\end{prop}

\begin{proof}
Let $w\in W\left(M;q\right)$. Since $h$ is an isometry we have
\begin{eqnarray*}
\mathrm{Hess}\left(w\circ h\right) & = & h^{*}\left(\mathrm{Hess}w\right)\\
& = & \left(w\circ h\right)\left(h^{*}q\right)\\
& = & \left(w\circ h\right)q.
\end{eqnarray*}
This shows that $h^{*}:W\left(M;q\right)\rightarrow W\left(M;q\right)$. Next we note that
\begin{eqnarray*}
\bar{\mu}\left(w\circ h\right) & = & \left(\kappa\circ h\right)\left(w\circ h\right)^{2}+\left|\nabla\left(w\circ h\right)\right|^{2}\\
& = & \left(\kappa\circ h\right)\left(w\circ h\right)^{2}+\left|Dh^{-1}\left(\left(\nabla w\right)\circ h\right)\right|^{2}\\
& = & \left(\kappa\circ h\right)\left(w\circ h\right)^{2}+\left|\left(\nabla w\right)\circ h\right|^{2}\\
& = & \bar{\mu}\left(w\right)\circ h
\end{eqnarray*}
which shows that $\bmu$ is preserved by $h$.
\end{proof}

Let $\Gamma\subset\mathrm{Iso}\left(M,g\right)$ be a subgroup that preserves the quadratic form $q$. Define
\[
W\left(M;q,\Gamma\right)=\left\{ w\in W\left(M;q\right):w\circ h=w,\, \text{for all } h\in\Gamma\right\} \subset W\left(M;q\right)
\]
as the fixed point set of the action of $\Gamma$ on $W\left(M;q\right)$. As  $W=W\left(M;q,\Gamma\right)$ is a subspace of $W(M;q)$, when $M$ is simply connected,  we can apply Theorems \ref{thm:WPHWP}, \ref{thm:W-S-notempty} and \ref{thm:W} to obtain a warped product splitting
\[ M = B \times_u F \]
such that
\[
W\left(M;q,\Gamma\right)=\left\{ \pi_{1}^{*}(u)\cdot\pi_{2}^{*}(v):v\in W\left(F;-\bar{\mu}\left(u\right)g_{F}\right)\right\} .
\]
Moreover,  $\Gamma$ will preserve the foliations $\mathcal{F}$ and $\widehat{\mathcal{F}}$ defined by the subspace $W\left(M;q,\Gamma\right)$ since
\[
Dh^{-1}\left(\left(\nabla w\right)|_{h}\right)=\nabla\left(w\circ h\right)=\nabla w.
\]
Thus $\Gamma$ preserves the distributions and fixes $u_{p}$. In particular, it induces an action on $F$ and an action  on $B$ that leaves $u$ as well as $q_{B}$ invariant.

By considering the different cases for $F$, we obtain the following.
\begin{prop}\label{prop:Gammaaction}
Let $\left(M,g\right)$ be complete and simply connected and assume that $\dim W\left(M;q,\Gamma\right)=k+1\geq2$. Then either
\begin{enumerate}
\item  $\Gamma$ acts trivially on $F$, or
\item $F = \mathbb{R}$ and $\Gamma$ acts via translation on $\mathbb{R}$.
\end{enumerate}
Moreover, if $W\left(M;q,\Gamma\right)$ contains a positive function, then $\Gamma$ acts trivially on $F$.
\end{prop}

\begin{proof}
We note that $\Gamma$ must leave all elements in $W\left(F;-\bar{\mu}\left(u\right)g_{F}\right)$ invariant. However, a close inspection of all cases shows that the only situation where a nontrivial subgroup of $\mathrm{Iso}\left(F,g_{F}\right)$ fixes all elements in $W\left(F;-\bar{\mu}\left(u\right)g_{F}\right)$ is when $F=\mathbb{R}$ and  all solutions are periodic.

If $W\left(M;q,\Gamma\right)$ contains a positive function, then so must $W\left(F;-\bar{\mu}\left(u\right)g_{F}\right)$.  However, when $F=\mathbb{R}$ it is not possible for  $W\left(F;-\tau dt^2\right)$ to contain a positive function and to have all functions  periodic.  To see this, let $v_1$ and $v_2$ be linearly independent solutions  to $v'' = -\tau v$  and assume $v_1$ is positive.  Then
\[ \left( \frac{v_2}{v_1} \right) ' =  \frac{W}{v_1^2} \]
where $W = v_2' v_1 -v_2 v_1'$ is the Wronskian, which, in this case,  is a non-zero constant.  In particular,  $v_2/v_1$ is strictly monotone, showing that $v_1$ and  $v_2$ can not both be periodic.
\end{proof}

\begin{rem}  \label{rem:PeriodicExamples}
Case (2) does occur.  The simplest example is when  $\bar{\mu} =\tau dt^{2}$   where $\tau$ is a positive constant. Since then
\[
W\left(\mathbb{R};-\tau dt^{2}\right)= A\sin(\sqrt{\tau} t) + B \cos(\sqrt{\tau} t)
\]
and so all solutions are invariant under a translation of length $\frac{2\pi}{\sqrt{\tau}}$.    There are also examples where  $\bar{\mu}$  is not constant. This is related to the problem of finding coexisting solutions to Hill's equation. Specifically, it is possible to choose $\tau\left(t\right)$ as in Example \ref{Exa:F=R} to be periodic with period $2\pi$ and such that the solutions space $W\left(\mathbb{R};-\tau dt^{2}\right)$ consists of $2\pi$ periodic functions (see \cite[Chapter 7]{Magnus-Winkler}.)
\end{rem}

When $\Gamma$ acts trivially on $F$, the  following corollary follows easily from previous results.

\begin{cor}
If $\Gamma$ acts trivially on $F$ and properly on $M$ with only principal isotropy, then
\[
\left(M/\Gamma,g\right)=\left(B/\Gamma\right)\times_{u}F
\]
and
\[
\pi^{*}\left(W\left(M/\Gamma;q\right)\right)=W\left(M;q,\Gamma\right)
\]
where $\pi:M\rightarrow M/\Gamma$ is the quotient map.
\end{cor}

From these results  we can now extract information about the case where $M$ is not simply connected. In that case we obtain a covering map $\pi:\widetilde{M}\rightarrow M$ and can think of $\Gamma=\pi_{1}\left(M\right)$ as acting properly by isometries on the universal covering $\widetilde{M}$. Moreover this action will clearly preserve the pull back of any quadratic form on $M$.  Applying the results to this case gives the following.

\begin{prop}\label{prop:nonsimplyconn}
Assume that $\left(M,g\right)$ is complete and that $q$ is a quadratic form on $M$ with $\dim W\left(M;q\right)=k+1\geq2$.   Let $\widetilde{M}=B\times_{u}F$ be the warped product splitting  on   the universal cover of $M$   coming from $W\left(M;q,\Gamma\right)$ with $\Gamma=\pi_{1}\left(M\right)$.    Then either
\begin{enumerate}
\item $k=1$,  $F = \mathbb{R}$, and $M$ is isometric to a  quotient of  $B \times_{u}  \mathbb{R}$ by a diagonal action of $\Gamma$ that preserves $u$ on B and acts via translations on $\mathbb{R}$, or
\item we have
\[ \left(M,g\right)=\left(B/\Gamma\right)\times_{u}F.\]
\end{enumerate}
In either case we also have
\[
\pi^{*}\left(W\left(M;q\right)\right)=W\left(\widetilde{M};q,\Gamma\right) = \left\{ \pi_{1}^{*}(u)\cdot\pi_{2}^{*}(v):v\in W\left(F;-\bar{\mu}\left(u\right)g_{F}\right)\right\}.
\]
Moreover, if $W\left(M;q\right)$ contains a positive function then we have (2).
\end{prop}

\begin{rem}
Taking $F$ to be one of the examples mentioned in  Remark \ref{rem:PeriodicExamples}  produces examples  showing case (1) does occur.   These examples are locally warped products, since the induced action on $B$ must preserve $u$, but they do not admit a global warped product structure.
\end{rem}

\medskip
\section{Proof of Theorem C}

In this section we prove Theorem C, the uniqueness of warped product metrics under Ricci and scalar curvature assumptions. Recall that $(M^n, g)$ is a complete Riemannian manifold and $w_1$, $w_2$ are two positive functions on $M$. The warped product manifolds $E_1$ and $E_2$ are given by
\begin{eqnarray*}
E_1 = M \times_{w_1} N_1, & & E_2 = M\times_{w_2} N_2
\end{eqnarray*}
where $(N^d_i, h_i)(i=1,2)$ are simply-connected space forms. We assume that $E_1$, $E_2$ have the same scalar curvature and the following condition on Ricci curvatures
\[
\Ric^{E_1}(X, Y) = \Ric^{E_2}(X, Y)
\]
for any $X, Y \in TM$.

Let us briefly describe the main steps in the proof. The condition on Ricci curvatures restricted to $TM \subset TE_i$ shows that there is a quadratic form $q$ on $M$ such that $w_1, w_2$ solve the system $\Hess w = w q$. In particular, $\dim W(M; q) = k+1 \geq 2$ as $w_1$ and $w_2$ are linearly independent. Since $w_1$ and $w_2$ are positive functions, we are in Case (2) of Proposition \ref{prop:nonsimplyconn}. This implies that $M$ is a warped product on a base space $(B, g_B, u)$. It follows that the warped product structures of $E_1$ and $E_2$ have refinements over the the base space $B$ and larger fibers $(F^{k+d}_1, g_1)$ and $(F^{k+d}_2, g_2)$. When $k + d\geq 3$, or $k+d = 2$ and $F_1$, $F_2$ have constant Gauss curvature, we show that $F_1$, $F_2$ are simply-connected space forms with the same curvature. Thus they are isometric and we also obtain an isometry between $E_1$ and $E_2$. This finishes case (1) in the statement of Theorem C. In the exceptional case where $k=d=1$, we show that $F_1$, $F_2$  are diffeomorphic to $\Real^2$ and have the same varying Gauss curvature. This gives case (2) in Theorem C.

\begin{proof}[Proof of Theorem C]
We assume that $w_{1}$ and $w_{2}$ are linearly independent and let $\kappa_i$ be the sectional curvature of $(N_i, h_i)$ for $i=1,2$.

For any two vectors $X, Y\in TM$ we have
\[
\mathrm{Ric}^{E_{i}}(X, Y)=\mathrm{Ric}(X,Y)-\frac{d}{w_{i}}\mathrm{Hess}w_{i}(X,Y).
\]
Since $\mathrm{Ric}^{E_{1}}(X,Y)=\mathrm{Ric}^{E_{2}}(X,Y)$ we have
\[
\frac{1}{w_{1}}\mathrm{Hess}w_{1}(X,Y)=\frac{1}{w_{2}}\mathrm{Hess}w_{2}(X,Y)=q(X,Y)
\]
where
\begin{equation*}
q(X,Y)=\frac{1}{d}\left(\mathrm{Ric}-\mathrm{Ric}^{E_{i}}\right)(X,Y)
\end{equation*}
is a smoothly varying quadratic form on $(M,g)$. So the following vector space $W$ has dimension at least $2$
\[
W(M; q)=\left\{ w\in C^{\infty}(M, g):\mathrm{Hess}w=wq\right\} .
\]
Let $k=\dim W(M;q) -1\geq 1$. It follows from Proposition \ref{prop:nonsimplyconn} that $(M,g)$ is a warped product metric
\begin{equation}\label{eq:MBF}
M=B^{b}\times_{u}F^{k}
\end{equation}
and $w_{i}=u\cdot v_{i}$$(i=1,2)$ where each $v_{i}\in C^{\infty}(F)$ satisfies the following Hessian equation on $(F,g_{F})$
\begin{equation}\label{eq:Hessvi}
\mathrm{Hess}_{F} v_{i}=-\tau v_{i}g_{F}.
\end{equation}
In particular we have
\[
\Delta_F v_i = - k \tau v_i.
\]

Let
\begin{eqnarray*}
F_{1}=F\times_{v_1} N_{1} & \text{with} & g_{1}=g_{F}+v_{1}^{2}h_{1}\\
F_{2}=F\times_{v_2} N_{2} & \text{with} & g_{2}=g_{F}+v_{2}^{2}h_{2}.
\end{eqnarray*}
From the warped product decomposition (\ref{eq:MBF}) the manifolds $E_{1}$ and $E_{2}$ can be written using the new base $(B, g_B)$ and fibers $(F_1, g_1)$, $(F_2, g_2)$:
\begin{eqnarray*}
E_{1}=B\times_{u}F_1 & \text{with} & g_{E_{1}}=g_{B}+u^{2}g_1 \\
E_{2}=B\times_{u}F_2 & \text{with} & g_{E_{2}}=g_{B}+u^{2}g_2.
\end{eqnarray*}
In the following we show that $(F_1, g_1)$ and $(F_2, g_2)$ are isometric unless they are complete surfaces with the same varying Gauss curvature.

\smallskip

In terms of the base $B$ and fiber $F_i$, the scalar curvatures of $E_{1}$ and $E_{2}$ are given by
\begin{eqnarray*}
\mathrm{scal}^{E_{1}} & = & \mathrm{scal}^{B}+\frac{1}{u^{2}}\left(\mathrm{scal}^{F_{1}}-2(k+d)u\Delta_{B}u-(k+d)(k+d-1)\left|\nabla u\right|^{2}\right)\\
\mathrm{scal}^{E_{2}} & = & \mathrm{scal}^{B}+\frac{1}{u^{2}}\left(\mathrm{scal}^{F_{2}}-2(k+d)u\Delta_{B}u-(k+d)(k+d-1)\left|\nabla u\right|^{2}\right).
\end{eqnarray*}
Note that $(N_i, h_i)$ has constant sectional curvature $\kappa_i$ for $i=1,2$. The scalar curvature of $(F_1, g_1)$ is given by
\begin{eqnarray*}
\scal^{F_1} & = & \scal^F + \frac{1}{v_1^2}\left(\scal^{N_1} - 2d v_1 \Delta_{F}v_1 - d(d-1)\abs{\nabla v_1}^2\right) \\
& = & \scal^F + \frac{d}{v_1^2}\left((d-1)\kappa_1 +2k \tau v_1^2 -(d-1)\abs{\nabla v_1}^2 \right) \\
& = & \scal^F + \frac{d(d-1)}{v_1^2}\left(\kappa_1 -\abs{\nabla v_1}^2\right) + 2k \tau d.
\end{eqnarray*}
Similarly we have
\begin{equation*}
\scal^{F_2} = \scal^F + \frac{d(d-1)}{v_2^2}\left(\kappa_2 -\abs{\nabla v_2}^2\right) + 2k \tau d.
\end{equation*}
Since $d\geq 1$, the condition $\scal^{E_1} = \scal^{E_2}$ is equivalent to the following identity on $(F, g_F)$:
\begin{equation}\label{eq:v1v2}
\frac{d-1}{v_1^2}\left(\kappa_1 - \abs{\nabla v_1}^2\right) = \frac{d-1}{v_2^2}\left(\kappa_2 - \abs{\nabla v_2}^2\right).
\end{equation}

Next we compute the Ricci curvatures of $(F_i, g_i)(i=1,2)$. We follow the convention that $X, Y,...$ represents the horizontal vector fields and $V, W, ...$ the vertical ones for the Riemannian submersions $(F_i, g_i) \rightarrow (F, g_F)$. Note that $\Ric^{F_i}(X, V) = 0$. For the nonvanishing term we have
\begin{eqnarray*}
\Ric^{F_i}(X, Y) & = & \Ric^F(X, Y) - \frac{d}{v_i}\Hess_F v_i(X, Y) \\
& = & \Ric^F(X,Y) - \frac{d}{v_i}(-\tau v_i) g_F(X, Y) \\
& = & \Ric^F(X, Y) + \tau d \, g_F(X, Y),
\end{eqnarray*}
and
\begin{eqnarray*}
\Ric^{F_i}(V, W) & = & \Ric^{N_i}(V, W) - v_i^2 h_i(V, W)\left(\frac{\Delta_{F}v_i}{v_i} + (d-1)\frac{\abs{\nabla v_i}^2}{v_i^2}\right) \\
& = & \frac{(d-1)\kappa_i}{v_i^2} g_i(V, W) - g_i(V, W)\left(-k \tau + (d-1)\frac{\abs{\nabla v_i}^2}{v_i^2}\right) \\
& = & g_i(V,W)\left(k \tau + \frac{d-1}{v_i^2}\left(\kappa_i - \abs{\nabla v_i}^2\right)\right).
\end{eqnarray*}
Thus we have the following non-vanishing Ricci curvatures
\begin{eqnarray}
\Ric^{F_i}(X, Y) & = & \Ric^F(X, Y) + \tau d \, g_F(X, Y) \label{eq:RicFiXY} \\
\Ric^{F_i}(V, W) & = & g_i(V, W)\left(k \tau + \frac{d-1}{v_i^2}\left(\kappa_i - \abs{\nabla v_i}^2\right)\right) \label{eq:RicFiVW}
\end{eqnarray}

\smallskip

\textsc{Case A.} We assume that $k\geq 2$. In this case $\tau$ is a constant and $(F, g_F)$ is a simply-connected space form with sectional curvature $\tau$. Using the equation $\Hess_F v_i = -\tau v_i g_F$ we have
\[
\nabla\left(\tau v_i^2 + \abs{\nabla v_i}^2\right) = 2\tau v_i \nabla v_i + 2\Hess_F v_i(\nabla v_i) = 0
\]
which shows that
\begin{equation}\label{eq:mui}
\mu_i = \tau v_i^2 + \abs{\nabla v_i}^2
\end{equation}
is a constant on $F$. So the identity (\ref{eq:v1v2}) is equivalent to
\[
\frac{d-1}{v_1^2}\left(\kappa_1 +\tau v_1^2 - \mu_1\right) = \frac{d-1}{v_2^2}\left(\kappa_2 +\tau v_2^2 - \mu_2\right)
\]
i.e.,
\[
\frac{d-1}{v_1^2}(\kappa_1 - \mu_1) = \frac{d-1}{v_2^2}(\kappa_2 - \mu_2).
\]
Since $v_1$ and $v_2$ are linearly independent, we conclude that
\begin{equation}\label{eq:kappamu0}
(d-1)(\kappa_1- \mu_1) = (d-1)(\kappa_2 - \mu_2) = 0.
\end{equation}

We consider the Ricci curvatures of $(F_i, g_i)$ in this case. Using the equation (\ref{eq:mui}) of $\mu_i$ the Ricci curvatures in equations (\ref{eq:RicFiXY}) and (\ref{eq:RicFiVW}) can be simplified as
\begin{eqnarray*}
\Ric^{F_i}(X, Y) & = & (k+d-1)\tau g_i (X, Y) \\
\Ric^{F_i}(V, W) & = & g_i(V, W)\left(k \tau + \frac{d-1}{v_i^2}\left(\kappa_i + \tau v_i^2 - \mu_i\right)\right) \\
& = & (k+d-1)\tau g_i(V,W) + \frac{(d-1)(\kappa_i - \mu_i)}{v_i^2}g_i(V, W) \\
& = & (k+d-1)\tau g_i(V,W).
\end{eqnarray*}
In the last equality we used the identity (\ref{eq:kappamu0}). It follows that both $(F_1, g_1)$ and $(F_2, g_2)$ are $(k+d-1)\tau$-Einstein manifolds. So $(F_1, g_1)$ and $(F_2, g_2)$ are isometric to each other by Lemma \ref{lem:spaceformEinstein} which is stated after this proof.

\smallskip

\textsc{Case B.} We assume that $k=1$. We have the following two cases:
\begin{enumerate}
\item[(B.1) ] $\tau(t)$ is a constant function.
\item[(B.2) ] $\tau(t)$ is not a constant function.
\end{enumerate}

\smallskip

In both cases the Hessian equation (\ref{eq:Hessvi}) for $v=v(t)$ reduces to
\begin{equation}\label{eq:vddt}
v''(t) + \tau v(t) = 0.
\end{equation}

In Case (B.1) from examples in section 1, we know that $\tau < 0$, $F = \Real^1$ and both $(F_1, g_1)$ and $(F_2, g_2)$ are hyperbolic spaces with sectional curvature $\tau$. So they are isometric.

In Case (B.2), we have

\begin{claim}
If $\tau(t)$ is not a constant function, then $d=1$.
\end{claim}

We argue by contradiction. Specifically, assume that  $d\geq 2$ and $\tau(t)$ is not constant.
Identity (\ref{eq:v1v2}) from the condition $\scal^{E_1} = \scal^{E_2}$ reduces to
\begin{equation}\label{eq:v1v1Real1}
\frac{d-1}{v_1^2}(\kappa_1 - (v_1'(t))^2) = \frac{d-1}{v_2^2}(\kappa_2 - (v_2'(t))^2).
\end{equation}
The Ricci curvatures on $(F_i, g_i)$ have the following form
\begin{eqnarray}
\Ric^{F_i}(\partial_t, \partial_t) & = & \tau d \label{eq:RicFiXYReal1}\\
\Ric^{F_i}(V, W) & = & g_i (V, W)\left(\tau + \frac{d-1}{v_i^2}\left(\kappa_i - (v_i'(t))^2\right)\right). \label{eq:RicFiVWReal1}
\end{eqnarray}
Let $f_i(t) = \log v_i(t)$($i=1,2$) and none of them is a constant. Otherwise $\tau(t)$ would be equal to zero by the differential equation (\ref{eq:vddt}). The Ricci curvature in (\ref{eq:RicFiVWReal1}) can be written as
\begin{equation}\label{eq:RicVWf}
\Ric^{F_i}(V,W) = g_i(V,W)\left(\tau + (d-1)(\kappa_i e^{-2f_i} - (f_i'(t))^2)\right).
\end{equation}
Equations (\ref{eq:vddt}) and (\ref{eq:v1v1Real1}) are equivalent to
\begin{equation}\label{eq:fddtau}
f_i''(t) + (f_i'(t))^2 + \tau(t) = 0
\end{equation}
and
\begin{equation}\label{eq:f1f2}
\kappa_1 e^{-2f_1} - (f_1'(t))^2 = \kappa_2 e^{-2f_2} - (f_2'(t))^2.
\end{equation}
Differentiating equation (\ref{eq:f1f2}) yields
\[
-2\left(\kappa_1 e^{-2f_1}f_1'(t) + f_1'(t) f_1''(t)\right) = -2\left(\kappa_2 e^{-2f_2}f_2'(t) + f_2'(t) f_2''(t)\right)
\]
i.e.,
\begin{equation}\label{eq:fdfdd}
f_1'(t)\left(\kappa_1 e^{-2f_1} + f_1''(t)\right) = f_2'(t)\left(\kappa_2 e^{-2f_2} + f_2''(t)\right).
\end{equation}
On the other hand adding equation (\ref{eq:fddtau}) to equation (\ref{eq:f1f2}) yields
\[
\kappa_1 e^{-2f_1} + f_1''(t) + \tau(t) = \kappa_2 e^{-2f_2} + f_2''(t) + \tau(t),
\]
i.e.,
\[
\kappa_1 e^{-2f_1} + f_1''(t) = \kappa_2 e^{-2f_2} + f_2''(t).
\]
Comparing the last equation with equation (\ref{eq:fdfdd}) implies that either $f_1'(t) = f_2'(t)$ or
\begin{equation}\label{eq:kappaf}
\kappa_1 e^{-2f_1} + f_1''(t) = \kappa_2 e^{-2f_2} + f_2''(t) = 0.
\end{equation}
If $f'_1(t) = f'_2(t)$ on an open interval then $v_1 = C v_2$ on that interval. So the Wronksian of $v_1$ and $v_2$ is zero on that interval. Since $v_1$ and $v_2$ satisfy the same equation $v''(t) + \tau(t) v(t) = 0$, the Wronskian is constant, and so it is zero everywhere that contradicts the linear independence assumption of $v_1$ and $v_2$. It follows that  $f_1'(t) \ne f_2'(t)$ almost everywhere implying equations in (\ref{eq:kappaf}) hold everywhere by continuity. Subtracting equation (\ref{eq:fddtau}) from equations (\ref{eq:kappaf}) yields
\[
\kappa_i e^{-2f_i} - (f_i'(t))^2 = \tau(t).
\]
Plugging this in to the Ricci curvature in (\ref{eq:RicVWf}) shows that $\Ric^{F_i}(V, W) = \tau d \, g_i(V,W)$. Combining with the fact $\Ric^{F_i}(\partial_t, \partial_t) = \tau d$ we see that the Ricci curvature of $(F_i, g_i)$ is $\tau d$ and thus does not depend on the vector we choose. From Schur's lemma it follows that $\tau(t)$ is a constant function as $\dim F_i = d+1 \geq 3$. This gives us the desired contradiction.

So we have $d=1$, $F_1$ and $F_2$ are diffeomorphic to $\Real^2 =\set{(t,x)}$ and both metrics $g_1$ and $g_2$ have the same Gauss curvature $\tau(t)$. From Example \ref{ex:surfacenonisometric} there are such surfaces with metrics $g_i = dt^2 + v^2_i(t) dx^2$($i=1,2$) that are not isometric. This gives us Case (2).

Except for examples in case (2), we know that $(F_1, g_1)$ is isometric to $(F_2, g_2)$. So the isometry between $E_1$ and $E_2$ follows easily. This finishes the proof.
\end{proof}

Next we state and prove the lemma used in the proof of Theorem C.  

\begin{lem}\label{lem:spaceformEinstein}
Let $(F^k, g_F)$ be a simply-connected space form with two positive functions $v_1, v_2 \in C^\infty(F)$. Let $(N_1^d, h_1)$ and $(N_2^d, h_2)$ be two simply-connected space form with $k+d \geq 3$. Suppose $(F_1, g_1)$ and $(F_2, g_2)$ are two warped product manifolds as follows
\begin{eqnarray*}
F_1 = F \times N_1 & \text{with} & g_1 = g_F + v_1^2 h_1 \\
F_2 = F \times N_2 & \text{with} & g_2 = g_F + v_2^2 h_2.
\end{eqnarray*}
If both $(F_1, g_1)$ and $(F_2, g_2)$ are Einstein manifolds with the same Einstein constant, then they are isometric to each other.
\end{lem}

The above lemma is a special case of a more general uniqueness result that two $\lambda$-Einstein metrics on warped products with a fixed base manifold and simply-connected space form fibers are isometric. We give a proof in this special case that does not appeal to the general result.

\begin{proof}
Note that $F_1$ and $F_2$ are isometric if $v_1$ and $v_2$ are linearly dependent. In the following we assume that they are linearly independent. Let $\kappa_i$ be the sectional curvatures of $(N_i, h_i)$ and $(k+d-1)\tau$ the Einstein constant of $(F_i, g_i)$ for $i=1,2$.

We consider the case $k\geq 2$ first. For any vectors $X, Y\in TF$ we have
\[
\Ric^{F_i}(X, Y) = \Ric^F (X, Y) - \frac{d}{v_i}\Hess v_i(X, Y) = (k+d-1)\tau g_F(X, Y),
\]
i.e.,
\[
\Hess v_i = \frac{v_i}{d}\left(\Ric^F - (k+d-1)\tau g_F\right)
\]
on $(F, g_F)$. The fact that the system above has more than one solution on a simply-connected space form implies that
\[
\Ric^F = (k-1)\tau g_F,
\]
see \cite[Example 2.1]{HPWVirt-Soln}. So the sectional curvature of $F$ is $\tau$ and
\[
\Hess v_i = - \tau v_i g_F.
\]
Using the formula of $\Ric^{F_i}(V,W)$ for vectors $V, W \in TN_i$ and the identity $\Delta_F v_i = -k \tau v_i$ we obtain
\[
\frac{d-1}{v_i^2}(\kappa_i - \abs{\nabla v_i}^2) = (d-1)\tau.
\]
Then the Riemannian curvature formulas of warped product metric \cite[Proposition 42 in Chapter 7]{O'Neill} implies that $(F_i, g_i)$ has constant curvature $\tau$. Since both $F_1$ and $F_2$ are simply-connected and have the same curvature, they are isometric.

In the case when $k=1$, we have $(F, g_F) = (\Real^1, dt^2)$ with $\Ric^F = 0$ and $v_i = v_i(t)$($i=1,2$). Then $\Ric^{F_i}(\partial_t, \partial_t) = \tau d$ yields the following differential equation
\[
v_i''(t) + \tau v_i(t) = 0.
\]
The fact that the above equation has more than one positive solutions implies that $\tau < 0$ and then $(F_1, g_1)$ and $(F_2, g_2)$ are simply-connected hyperbolic space with curvature $\tau$, see \cite[Example 1.2]{HPWVirt-Soln}. So we also conclude that $F_1$ and $F_2$ are isometric and this finishes the proof.
\end{proof}

We construct two non-isometric complete metrics on $\Real^2 = \set{(t,x)}$ with the same varying Gauss curvature $\tau(t)$.
\begin{example}\label{ex:surfacenonisometric}
Let $v_1(t)$ be a smooth positive function on $\Real$ and $g_1 = dt^2 + v_1(t)^2 dx^2$ a warped product metric on $\Real^2$. The Gauss curvature of $g_1$ at $(t,x)$ is given by
\[
\tau(t) = - \frac{v_1''(t)}{v_1(t)}.
\]
Finding another metric $g_2 = dt^2 + v_2(t)^2 dx^2$ with the same Gauss curvature is equivalent to finding another positive solution $v_2(t)$ with
\[
v''_2(t) + \tau (t)v_2(t) = 0.
\]
Let $v_2(t) = v_1(t) u(t)$ and then $u(t)$ can be solved as
\[
u(t) = C_1 \int_0^t\frac{1}{v^2_1(s)} ds + C_2
\]
for arbitrary constants $C_1$ and $C_2$. By rescaling we assume that $C_1 = 1$. Positivity of $v_2(t)$ implies $u(t) > 0$ for all $t$. Since $u(t)$ is monotone increasing this is equivalent to the following inequality
\begin{equation}\label{eq:uinfty}
u(-\infty) = \int_{0}^{-\infty} \frac{1}{v^2_1(s)} ds + C_2 \geq 0.
\end{equation}
So for any positive function $v_1(t)$ which satisfies the above inequality (\ref{eq:uinfty}) for some constant $C_2$ there is another metric $g_2 = dt^2 + \left(u(t)v_1(t)\right)^2 dx^2$ having the same Gauss curvature. For example, let $v_1(t) = e^{(t^2/2)}$. Then the Gauss curvature is $\tau(t) = -t^2 -1$ and $v_2(t) = u(t) v_1(t)$ with
\[
u(t) = \int_0^t e^{-s^2}ds + 1 = \frac{\sqrt{\pi}}{2}\mathrm{erf}(t) + 1 > 0,
\]
where $\mathrm{erf}(t)$ is the Gauss error function bounded below by $-1$.

Next we show that these two metrics are not isometric. Suppose not, then there is a diffeomorphism $\sigma : \Real^2 \rightarrow \Real^2$ sending $g_1$ to $g_2$, i.e., $\sigma^*(g_1) = g_2$. Fix an interval $I$ where $\tau'(t) \ne 0$ for all $t\in I$. Let $L_t = \set{(t,x) : x\in \Real}$ be the $t$-fiber and note that the second fundamental form of $L_t$ is given by $(\log v_i(t))'$. Since the Gauss curvature $\tau(t)$ is preserved by $\sigma$, the image of $L_t$ with $\tau'(t) \ne 0$ is another fiber $L_{f(t)}$ where $f(t)$ is a smooth function on $I$ and $\tau(f(t))= \tau(t)$. This means that the isometry must also preserve the horizontal directions. These directions are simply the curves $t \rightarrow (t,x)$. As these curves are horizontal geodesics they will be mapped to horizontal geodesics. In particular, $f(t)$ is it self an isometry. In our case this means that $f(t)=t$ or $f(t)=-t$. However as $g_1$ is invariant under $(t,x) \rightarrow (-t,x)$, we can assume that $f(t)=t$. As the isometry also preserves the second fundamental forms of the fibers it follows that
\[
\frac{v_1'(t)}{v_1(t)} = \frac{v_2'(t)}{v_2(t)}\quad \text{for}\quad t \in I,
\]
and thus $v_1(t)$ and $v_2(t)$ are linearly dependent on $I$. So they are linearly dependent on the whole $\Real$ as they solve the same differential equation $v''(t) + \tau(t) v(t) = 0$. This contradicts the linear independency assumption and thus $g_1$ and $g_2$ are not isometric.

Note that the same analysis works for any two $v_1$ and $v_2$ solving $v'' + \tau (t)v = 0$ where $\tau (t)$ is strictly monotone.
\end{example}

\begin{rem}\label{rem:surfacenonisometric}
Let $(B, g_B, u)$ be a base space with $u$ a constant function, i.e., the quadratic form $q = 0$. Let $\Sigma_1$ and $\Sigma_2$ be the non-isometric isocurved complete surfaces in Example \ref{ex:surfacenonisometric}. Then $E_1 = B\times \Sigma_1$ and $E_2 = B \times \Sigma_2$ with product metrics are not isometric. Otherwise the universal covers $\tilde{E}_i = \tilde{B}\times \Sigma_i$ would have two different de Rham factorizations which contradicts de Rham decomposition Theorem \cite{KobayashiNomizu}.
\end{rem}

\medskip
\section{Miscellaneous results}
In this section we present some related results. In subsection 8.1 we study the base space in more detail. In particular, we extend $\bmu$ to all functions in $W$. This extension defines a bilinear form on $W$ for many examples of $(M,g)$ that we are interested in. In subsection 8.2 we show that the exterior square $\wedge^2 W$ has a Lie algebra structure using the bilinear form $\bmu$ and prove that the map $\iota$ in (\ref{eq:iota}) is a Lie algebra homomorphism.

\subsection{Base Space Structure}
Recall that for $w \in W$ we have
\[
\bmu(w) = \kappa w^2 + \abs{\nabla w}^2.
\]

\begin{prop}\label{prop:mu-construction}
If $w=\pi_{1}^{*}(u)\cdot\pi_{2}^{*}(v)\in W$, then
\[
\nabla\left(\kappa w^{2}+\left|\nabla w\right|^{2}\right)=\frac{w^{2}}{u^{2}}\nabla\left(\kappa u^{2}+\left|\nabla u\right|^{2}\right)
\]
and
\[
\kappa w^{2}+\left|\nabla w\right|^{2}=\bar{\mu}\left(u\right)v^{2}+\left|\nabla v\right|_{F}^{2}
\]
is a constant on $M$ when $k> 1$ or $S \ne \emptyset$.
\end{prop}

\begin{proof}
Note that $\nabla u$ and $\nabla v$ are horizontal and vertical vector fields respectively and so they are orthogonal. The first equality follows from the calculation
\begin{eqnarray*}
\nabla\left(\kappa w^{2}+\left|\nabla w\right|^{2}\right) & = & w^{2}\nabla\kappa+2\kappa w\nabla w+2wQ\left(\nabla w\right)\\
& = & w^{2}\nabla\kappa+2w\kappa g\left(\nabla w,\nabla u\right)\frac{\nabla u}{\left|\nabla u\right|^{2}}+2w\frac{g\left(\nabla w,\nabla u\right)}{\left|\nabla u\right|^{2}}Q\left(\nabla u\right)\\
& & + 2 w\kappa g (\nabla w, \nabla v)\frac{\nabla v}{\abs{\nabla v}^2} + 2 w \frac{g(\nabla w, \nabla v)}{\abs{\nabla v}^2} Q(\nabla v) \\
& = & \frac{w^{2}}{u^{2}}\left(u^{2}\nabla\kappa+2\kappa u\nabla u+2uQ_{B}\left(\nabla u\right)\right)+ 2 w\frac{g(\nabla w, \nabla v)}{\abs{\nabla v}^2}\left(\kappa \nabla v + Q(\nabla v)\right) \\
& = & \frac{w^{2}}{u^{2}}\nabla\left(\kappa u^{2}+\left|\nabla u\right|^{2}\right).
\end{eqnarray*}
In the third equality above we used Lemma \ref{lem:u-structure}. Theorem \ref{thm:verticalP} and the definition of $\kappa$ at the beginning of section 5 show that the term in $\nabla v$ vanishes and we obtained the last equality above. 

For the second note that if $w=\pi_{1}^{*}(u)\cdot\pi_{2}^{*}(v)$ for $v\in W\left(F;-\bar{\mu}\left(u\right)g_{F}\right)$ then
\[
\kappa w^{2}+\left|\nabla w\right|^{2}=\bar{\mu}\left(u\right)v^{2}+\left|\nabla v\right|_{F}^{2}
\]
defines a function on $F$. In the case when $\bmu(u)$ is constant on $M$, i.e., $k> 1$ or $S \ne \emptyset$, we have
\[
\nabla\left|\nabla v\right|_{F}^{2}=2\nabla_{\nabla v}\nabla v=-2v\bar{\mu}\left(u\right)\nabla v
\]
which shows that $\kappa w^2 + \abs{\nabla w}^2$ is constant on $F$ and thus on $M$.
\end{proof}

We saw in Proposition \ref{prop:WsplittingHWP-1} that it is also necessary to compute $\bar{\mu}\left(z\right)$ even though $z$ is not an element in $W$.
\begin{prop}\label{prop:bmu-z}
If $\dim W\left(M;q\right)=k+1\geq2$ and $\dim W\left(B;q_{B}\right)\geq2$, then
\[
\nabla\bar{\mu}\left(z,z\right)=\frac{z^{2}}{u^{2}}\nabla\bar{\mu}\left(u\right) + 2\frac{z}{u}\left(\kappa-\kappa_{B}\right)\left(u\nabla z-z\nabla u\right)
\]
and
\[
\nabla\bar{\mu}\left(u,z\right)=\frac{z}{u}\nabla\bar{\mu}\left(u\right) + \left(\kappa-\kappa_{B}\right)\left(u\nabla z-z\nabla u\right)
\]
where $\kappa_{B}$ is the $\kappa$ defined on $B$ using $ W\left(B;q_{B}\right)$.
\end{prop}

\begin{proof}
Let $z\in W\left(B;q_{B}\right)$ and $K=u\nabla z-z\nabla u$ be the corresponding Killing vector field. Since $q_B = \frac{1}{u}\mathrm{Hess}_B u$, the condition $z\in W\left(B; q_B\right)$ implies that
\begin{equation}\label{eqn:uz}
\nabla_{\nabla u}\nabla z = \frac{z}{u}\nabla_{\nabla u}\nabla u.
\end{equation}
Then we have
\begin{eqnarray*}
\nabla\bar{\mu}\left(u,z\right) & = & \left(\nabla\kappa\right)uz+\kappa\left(\nabla u\right)z+\kappa u\nabla z+\nabla_{\nabla z}\nabla u+\nabla_{\nabla u}\nabla z\\
& = & \left(\nabla\kappa\right)uz+2\kappa\left(\nabla u\right)z+\kappa K+2\frac{z}{u}\nabla_{\nabla u}\nabla u+\frac{1}{u}\nabla_{K}\nabla u\\
& = & \frac{z}{u}\nabla\left(\bar{\mu}\left(u\right)\right)+\kappa K+Q_{B}\left(K\right)\\
& = & \frac{z}{u}\nabla\left(\bar{\mu}\left(u\right)\right)+\kappa K-\kappa_{B}K\\
& = & \frac{z}{u}\nabla\left(\bar{\mu}\left(u\right)\right)+\left(\kappa-\kappa_{B}\right)\left(u\nabla z-z\nabla u\right).
\end{eqnarray*}
Note that the second equality above follows from the identity (\ref{eqn:uz}). Since $K$ is a vertical vector field in the warped product decomposition of $B$ induced by $W(B; q_B)$, Theorem \ref{thm:verticalP} applied to $(B, g_B)$ implies that $Q_B(K) = \left(\rho_B + \mathrm{tr}Q_B\right) K = - \kappa_B K$ which gives the fourth equality above. 

Similarly we have
\[
\nabla\left(\kappa z^{2}+\left|\nabla z\right|^{2}\right)=\frac{z^{2}}{u^{2}}\nabla\left(\kappa u^{2}+\left|\nabla u\right|^{2}\right)+2\frac{z}{u}\left(\kappa-\kappa_{B}\right)\left(u\nabla z-z\nabla u\right)
\]
which shows the desired identities.
\end{proof}

\begin{cor}\label{cor:basestructure}
Assume that $M$ is simply connected such that $\dim W\left(M;q\right)=k+1\geq2$ and $S=\emptyset$. If $\bar{\mu}\left(u\right)$ is constant and $\dim W\left(B;q_{B}\right)\geq2$, then $\kappa\neq\kappa_{B}$.
\end{cor}

\begin{proof}
In case $\bar{\mu}\left(u\right)$ is constant and $\kappa=\kappa_{B}$ the above Proposition \ref{prop:bmu-z} implies that $\bar{\mu}=\bar{\mu}_{B}$ defines a quadratic form on $W\left(B;q_{B}\right)$. If $\dim W\left(B;q_{B}\right)\geq2$ it is then possible to find $z\in W\left(B;q_{B}\right)-\mathrm{span}\left\{ u\right\} $ such that $\bar{\mu}\left(u,z\right)=0$. Proposition \ref{prop:WsplittingHWP-1} then implies that $w=z+uv\in W$ when $v\in W\left(F;-\bar{\mu}\left(u\right)g_{F}\right)$. But that contradicts Theorem \ref{thm:W}.
\end{proof}

\smallskip
\subsection{The Lie algebraic structure on $\wedge^2 W$}

Suppose that $\dim W = k+1$, thus $\dim \wedge^2 W = \frac{1}{2}k(k+1)$. In this subsection we assume that $k > 1$, i.e., $\wedge^2 W$ is not of dimension one. From Proposition \ref{prop:mu-construction} the bilinear form $\bmu(w)$ is a constant on $M$ for any $w\in W$. We associate an element $v\wedge w \in \wedge^2 W$ with the following endomorphism $L \in \mathfrak{gl}(W)$:
\[
L(x) = \bmu(w,x)v - \bmu(v,x) w,\quad \forall x \in W.
\]
We show that the association is faithful. Fix a point $p \in M$ there is a localization $\bmu_p$ on $\Real \times T_p M$ of $\bmu$ defined by
\[
\bmu_p(\alpha, v) = \kappa(p) \alpha^2 + \abs{v}^2
\]
for any $(\alpha, v)\in \Real \times T_pM$. So the evaluation map $w \mapsto (w(p), \nabla w|_p)$ in Proposition \ref{prop:injectionHWP} is an injective isometry from $(W, \bmu)$ to $(\Real\times T_pM , \bmu_p)$. In particular the bilinear form $\bmu$ has nullity at most one. Any element in $\wedge^2 W$ can be written as $w_1\wedge w_2 + ... + w_{2s-1} \wedge w_{2s}$ where $w_1 , ..., w_{2s} \in W$ are linearly independent. If this element is mapped to zero in $\mathfrak{gl}(W)$ then $\bmu(w_i , x) = 0$ for all $x\in W$. This shows that $w_i$ lies in the nullspace for $\bmu$ contradicting that they are linearly independent. So we can view $\wedge^2 W$ as a subspace of $\mathfrak{gl}(W)$ with the natural Lie bracket operation from $\mathfrak{gl}(W)$. Moreover since $\bmu(L(x), y) + \bmu(x, L(y)) = 0$ for any $x, y\in W$ we have the following inclusion
\[
\wedge^2 W \subset \mathfrak{so}(W, \bmu)
\]
where
\[
\mathfrak{so}(W, \bmu) = \set{L \in \mathfrak{gl}(W) : \bmu(Lx, y) + \bmu(x, Ly) = 0, \, \forall x, y \in W}.
\]
By counting the dimensions of the two spaces we have $\wedge^2 W = \mathfrak{so}(W, \bmu)$ unless $\bmu$ has nullity one. In that case the codimension of $\wedge^2 W$ in $\mathfrak{so}(W, \bmu)$ is one.

Recall the map $\iota : \wedge^2 W \rightarrow \mathfrak{iso}(M, g)$ in (\ref{eq:iota}) sending $v\wedge w$ to $v\nabla w - w \nabla v$ and that $\iota(\wedge^2 W) \subset \mathfrak{iso}(M, g)$ is a Lie subalgebra by Corollary \ref{cor:wedgeWliealg}. We show that the two Lie algebra structures on $\wedge^2 W$ are compatible.

\begin{prop}\label{prop:liealg}
The subspace $\wedge^2 W \subset \mathfrak{so}(W, \bmu)$ is a Lie subalgebra and the map $\iota : \wedge^2 W \rightarrow \mathfrak{iso}(M, g)$ is an injective Lie algebra homomorphism.
\end{prop}
\begin{proof}
Let $L_i$ be the endomorphisms associated with $z_i = v_i \wedge w_i$ for $i=1, 2$. Then for any $x\in W$ we have
\begin{eqnarray*}
[L_1, L_2](x)  & = & L_1(L_2(x)) - L_2(L_1(x)) \\
& = & L_1(\bmu(w_2, x)v_2 - \bmu(v_2, x)w_2) - L_2(\bmu(w_1, x)v_1 - \bmu(v_1, x)w_1) \\
& = & \bmu(w_2, x)\left(\bmu(w_1, v_2)v_1 - \bmu(v_1, v_2)w_1\right) - \bmu(v_2, x)\left(\bmu(w_1, w_2)v_1 - \bmu(v_1, w_2)w_1\right) \\
& & - \bmu(w_1, x)\left(\bmu(w_2, v_1)v_2 - \bmu(v_2, v_1)w_2\right) + \bmu(v_1, x)\left(\bmu(w_2, w_1)v_2 - \bmu(v_2, w_1)w_2\right) \\
& = & \bmu(w_1, w_2)\left(\bmu(v_1, x)v_2 - \bmu(v_2, x) v_1\right) - \bmu(w_1,v_2)\left(\bmu(v_1, x)w_2 - \bmu(w_2, x)v_1\right) \\
& & - \bmu(w_2, v_1)\left(\bmu(w_1, x)v_2 - \bmu(v_2, x)w_1\right) + \bmu(v_1, v_2)\left(
\bmu(w_1, x)w_2 - \bmu(w_2, x)w_1\right) \\
& = & L_3(x)
\end{eqnarray*}
where $L_3$ is associated with
\[
z_3 = - \bmu(v_1, v_2)w_1 \wedge w_2 + \bmu(v_1, w_2)w_1 \wedge v_2 + \bmu(v_2, w_1)v_1 \wedge w_2 - \bmu(w_1, w_2)v_1 \wedge v_2 \in \wedge^2 W.
\]
This shows that $\wedge^2 W \subset \mathfrak{so}(W, \bmu)$ is a Lie subalgebra.

We already know that $\iota$ is an injection and only need to show that it is a homomorphism with respect to Lie algebra structures. Let $X_i = \iota(z_i)$ for $i= 1, 2, 3$. We have
\begin{eqnarray*}
X_3 & = & -\bmu(v_1, v_2)\left(w_1 \nabla w_2 - w_2 \nabla w_1 \right) + \bmu(v_1, w_2)\left(w_1 \nabla v_2 - v_2 \nabla w_1\right) \\
& & + \bmu(v_2, w_1)\left(v_1 \nabla w_2 - w_2 \nabla v_1 \right) - \bmu(w_1, w_2)\left(v_1 \nabla v_2 - v_2 \nabla v_1 \right) \\
& = & - w_1 g(\nabla v_1, \nabla v_2)\nabla w_2 + v_1 g(\nabla v_2, \nabla w_1)\nabla w_2 \\
& & + w_2 g(\nabla v_1, \nabla v_2)\nabla w_1 - v_2 g(\nabla v_1, \nabla w_2)\nabla w_1 \\
& & + w_1 g(\nabla v_1, \nabla w_2)\nabla v_2 - v_1 g(\nabla w_1, \nabla w_2)\nabla v_2 \\
& & - w_2 g(\nabla v_2, \nabla w_1)\nabla v_1 + v_2 g(\nabla w_1, \nabla w_2)\nabla v_1.
\end{eqnarray*}
We used the formula $\bmu(v, w) = \kappa v w + g(\nabla v, \nabla w)$ and all the terms without $g(\cdot, \cdot)$ factor cancel out in the last equality.

Note that $\mathrm{Hess} w = w q$ and we compute
\begin{eqnarray*}
\nabla_{X_1} X_2 & = & \nabla_{v_1 \nabla w_1 - w_1 \nabla v_1} \left(v_2 \nabla w_2 - w_2 \nabla v_2\right) \\
& = & v_1 \nabla_{\nabla w_1}\left(v_2 \nabla w_2 - w_2 \nabla v_2\right) - w_1\nabla_{\nabla v_1}\left(v_2 \nabla w_2 - w_2 \nabla v_2\right) \\
& = & v_1 g(\nabla w_1, \nabla v_2)\nabla w_2 + v_1 v_2 \mathrm{Hess} w_2(\nabla w_1) \\
& & - v_1 g(\nabla w_1, \nabla w_2)\nabla v_2 - v_1 w_2 \mathrm{Hess} v_2(\nabla w_1) \\
& & - w_1 g(\nabla v_1, \nabla v_2)\nabla w_2 - w_1 v_2 \mathrm{Hess} w_2(\nabla v_1) \\
& & + w_1 g(\nabla v_1, \nabla w_2)\nabla v_2 + w_1 w_2 \mathrm{Hess} v_2(\nabla v_1) \\
& = &  v_1 g(\nabla w_1, \nabla v_2)\nabla w_2 - w_1 g(\nabla v_1, \nabla v_2)\nabla w_2  \\
& & - v_1 g(\nabla w_1, \nabla w_2)\nabla v_2 + w_1 g(\nabla v_1, \nabla w_2)\nabla v_2.  \end{eqnarray*}
It follows that
\begin{eqnarray*}
[X_1, X_2] & = & v_1 g(\nabla w_1, \nabla v_2)\nabla w_2 - w_1 g(\nabla v_1, \nabla v_2)\nabla w_2 \\
& & - v_1 g(\nabla w_1, \nabla w_2)\nabla v_2 + w_1 g(\nabla v_1, \nabla w_2)\nabla v_2 \\
& & - v_2 g(\nabla w_2, \nabla v_1)\nabla w_1 + w_2 g(\nabla v_2, \nabla v_1)\nabla w_1 \\
& & + v_2 g(\nabla w_2, \nabla w_1)\nabla v_1 - w_2 g(\nabla v_2, \nabla w_1)\nabla v_1 \\
& = & X_3.
\end{eqnarray*}
This shows that $\iota$ is a homomorphism and finishes the proof.
\end{proof}

\medskip

\end{document}